\documentclass[a4paper]{amsart}

\usepackage{xifthen}
\usepackage{xcolor}
\newboolean{pedro}
\setboolean{pedro}{true}
\newcommand{\pedro}{\ifthenelse{\boolean{pedro}}{\color{blue}
    \setboolean{pedro}{false}}{\color{black}\setboolean{pedro}{true}}}

\newcounter{margin}

\newboolean{javier}
\setboolean{javier}{true}
\newcommand{\javier}{\ifthenelse{\boolean{javier}}{\color{red}\setboolean{javier}{false}}{\color{black}\setboolean{javier}{true}}}

\usepackage[foot]{amsaddr}
\usepackage{amssymb}
\usepackage{mathtools}
\usepackage{xparse}
\usepackage{pdfsync}

\theoremstyle{theorem}
\newtheorem{lemma}{Lemma}
\newtheorem{theorem}{Theorem}
\newtheorem{proposition}{Proposition}
\newtheorem{corollary}{Corollary}

\newtheorem*{lemma*}{Lemma}
\newtheorem*{theorem*}{Theorem*}
\newtheorem*{proposition*}{Proposition}
\newtheorem*{corollary*}{Corollary}
\newtheorem*{conjecture*}{Conjecture}

\theoremstyle{definition}
\newtheorem*{definition*}{Definition}
\newtheorem{definition}{Definition}

\newtheorem{remark}{Remark}
\newtheorem*{remark*}{Remark}
\newtheorem*{assumption*}{Assumption}

\newcommand{\ptt}[2]{\ensuremath{\frac{\partial #1}{\partial #2}}}
\newcommand{\pt}[1]{\ensuremath{\frac{\partial}{\partial #1}}}
\newcommand{\hot}{\ensuremath{{h.o.t.}}}

\newcommand{\abs}[1]{\ensuremath{\vert #1 \vert}}
\newcommand{\cc}{\ensuremath{(\mathbb{C}^{2},0)}}

\newcommand{\tsing}{truly singular}
\DeclareMathOperator{\ord}{ord}
\DeclareMathOperator{\tang}{tang}
\DeclareDocumentCommand{\ce}{g}{\ensuremath{\Gamma_{\epsilon  {\IfNoValueF{#1}{,#1}} }}}
\subjclass[2010]{32S05, 32S65, 14H20}

\title[Analytic moduli of plane branches and holomorphic flows]{Analytic moduli of plane branches and holomorphic flows}
\author{P. Fortuny Ayuso}
\email{fortunypedro@uniovi.es}
\address{Dpt. of Mathematics, Univ. of Oviedo, Spain.}
\author{J. Rib\'on}
\email{jribon@id.uff.br}
\address{Dpt. of Analysis, Univ. Federal Fluminense, Brazil.}
\date{\today}
\begin{document}

\begin{abstract}
  We study the behaviour (in the infinitesimal neighbourhood of the singularity)
  of a singular plane branch under the action of holomorphic flows. The
  techniques we develop provide a new elementary, geometric and dynamical
  solution to Zariski's moduli problem for singular branches in
  $({\mathbb C}^{2},0)$.  Furthermore, we study whether elements
  of the same class of analytic conjugacy are conjugated by a holomorphic
  flow; in particular we show that there
  exists an analytic class that is not complete: meaning that there are two
  elements of the class that are not analytically conjugated by a local
  diffeomorphism embedded in a one-parameter flow.
\end{abstract}
\maketitle
\section{Introduction}
The ``moduli problem for plane branches", as posed by Zariski \cite{Zariski4}
and recently solved in an algebraic way \cite{Hefez-Hernandes-classification},
has been found to have a particularly elementary solution when set in the
context of holomorphic flows \cite{PFA-moduli}. In this context, one may ask
whether two analytically equivalent branches are also equivalent under a
holomorphic flow, thus comparing the analytic and the holomorphic-flow
modulis. This leads in a natural way to studying how a plane singular branch
behaves under the action of holomorphic flows, which is the topic of the present
work.


Roughly speaking, in the analytic classification of plane branches, these are
reduced to normal form and two are equivalent if they share the same
normal form \cite{Hefez-Hernandes-classification}.  This normal
form is obtained by making coefficients of the Puiseux parametrization of the
curve equal to $0$ working jet by jet. The coefficients that may be turned into
$0$ are determined by the set of orders of contact of K\"{a}hler differentials
with the curve.  We show that such data is equivalent to
providing the set of orders of tangency of germs of holomorphic vector fields
defined in a neighborhood of $0$ in ${\mathbb C}^{2}$ with the curve (Corollary
\ref{cor:conductor-from-vector-and-form}).  Since vector fields are dynamical
objects that generate one-parameter groups, it is natural to consider
exponentials of germs of vector fields (with a certain order of tangency with
the curve) as normalizing transformations. This is the point of view of the
first author in \cite{PFA-moduli}.  It has been expanded in this work where the
relation between orders of tangency of local vector fields with a curve and the
reduction to normal form of its Puiseux parametrization is made explicit in
Theorem \ref{the:contact-exponent-is-contact-minus-n}.  As a consequence,
replacing the set of orders of contact of K\"{a}hler differentials with the set
of orders of tangency removes the need of interpreting the former set in
dynamical terms.

Let us be more precise.  Consider a singular holomorphic vector
field $X$ (so that $(0,0)$ is an equilibrium point of $X$)
defined in an open neighbourhood $U$ of $(0, 0) \in {\mathbb C}^{2}$ and an
irreducible germ of analytic curve $\Gamma$ contained in the same open set $U$,
say $\Gamma \equiv (f=0)$ for some
$f \in \mathcal {O}_{({\mathbb C}^{2},0)}$.  Let
${ \{ \psi_{s} \}}_{s \in {\mathbb C}}$ be the one-parameter group whose
infinitesimal generator is $X$. Let $\epsilon \in {\mathbb C}$; consider the
curve
\begin{equation*}
  \psi_{- \epsilon} (\Gamma) := \Gamma_{\epsilon} \equiv (f \circ
  \psi_{\epsilon} (x,y) =0) .
\end{equation*}
By expanding $f \circ \psi_{\epsilon}$ as a Taylor power series in the variable
$\epsilon$ we obtain
\begin{equation}
  \label{equ:taylor} \Gamma_{\epsilon} \equiv \left( \sum_{n=0}^{\infty}
    \frac{\epsilon^n}{n!} X^{n}(f) (x,y) =0 \right)
\end{equation}
where $X^{0} (f)= f$ and we define $X^{j+1}(f) = X(X^{j}(f))$ for $j \geq 0$
recursively. We shall call $\{\Gamma_{\epsilon}\}$ the \emph{holomorphic
  deformation of $\Gamma$ by $X$} (or by $\{\psi_\epsilon\}$).  The
coefficient of $x^{i} y^{j}$ for $f \circ \psi_{\epsilon}$ is an entire function
of $\epsilon$ for any $i+j \geq 0$.  Assume for simplicity that the tangent cone
of $\Gamma$ is not $x=0$. The curve $\Gamma_\epsilon$ has a Puiseux
parametrization of the form $(t^{n}, \sum_{j=n}^{\infty} a_j(\epsilon) t^j)$ for
any $\epsilon$ in a small neighborhood of $0$ in ${\mathbb C}$ where $n$ is the
multiplicity of $\Gamma$ at $(0,0)$.  Assume also that $\Gamma$ is not invariant
by $X$ and let $k$ be the first index such that $a_k (\epsilon)$ is not a
constant function. We denote $(X,\Gamma)_{(0,0)}=k$. The combination of
Corollary \ref{cor:conductor-from-vector-and-form} and Theorem
\ref{the:contact-exponent-is-contact-minus-n} implies
\begin{equation}
  \label{equ:tan}
  k = (X,\Gamma)_{(0,0)} = (X(f), f)_{(0,0)} - n - c +1
\end{equation}
where $(X(f), f)_{(0,0)}$ is the intersection multiplicity of
$f$ and $X(f)$, or in other words the tangency order of $X$ with $\Gamma$, and
$c$ is the conductor of $\Gamma$.  As a consequence, the reduction to normal
form depends in a straightforward way on the set of tangency orders of
holomorphic vector fields with the curve $\Gamma$.

 Property (\ref{equ:tan}) is not obvious: a
priori the value of $(X, \Gamma)_{(0,0)}$ could have depended on
other terms of the Taylor power series expansion of $f \circ \psi_{\epsilon}$.
As an example of a situation in which further terms of the Taylor power series
expansion are relevant, consider the intersection multiplicity
$(\Gamma, \Gamma_{\epsilon})_{(0,0)}$.  It is equal to
$\min \{ (X^n (f), f)_{(0,0)} : n \geq 1 \}$ for $\epsilon \in {\mathbb C}^{*}$
in a small neighborhood of $0$ by Equation (\ref{equ:taylor}). The minimum may
be realized for $n >1$ as is the case for $\Gamma = (y^2-x^3=0)$
and $X = x \frac{\partial}{\partial y}$ where $(X(f), f)_{(0,0)} =5$ and
$(X^2(f), f)_{(0,0)} = 4$.

The previous discussion motivates the study of the action of one-parameter
groups on irreducible curves.  Consider an equivalence class ${\mathcal C}$ for
the equivalence relation given by the analytic conjugacy of plane branches.  We
say that two curves $\Gamma_1, \Gamma_2 \in {\mathcal C}$ are {\it connected by
  a geodesic} if they are conjugated by the time $1$ flow $\mathrm{exp} (X)$ of
a germ of holomorphic singular vector field.  We say that ${\mathcal C}$ is {\it
  complete} if given any two curves $\Gamma_1, \Gamma_2 \in {\mathcal C}$ they
are connected by a geodesic.  The term complete is motivated by analogy with the
case of finite dimensional Lie groups $G$ that have a bi-invariant metric where
geodesics are of the form $t \mapsto \mathrm{exp} (tX) \cdot g$ where $X$
belongs to the Lie algebra of $G$, $g \in G$ and $t$ varies in ${\mathbb R}$.
An example of a complete class $\mathcal{C}$ is the class of smooth curves
(Proposition \ref{pro:smooth-is-complete}).

A priori, we could define a notion of formal completeness in which $X$ is a
formal vector field, i.e.  a derivation of ${\mathbb C}[[x,y]]$ that preserves
its maximal ideal.  The definitions are, in fact, equivalent.

\begin{theorem}
  \label{teo:forcomplete} Let ${\mathcal C}$ be a class of analytic conjugacy of
  plane branches. Then ${\mathcal C}$ is complete if and only if ${\mathcal C}$
  is formally complete.
\end{theorem}

The analytic classification \cite{Hefez-Hernandes-classification} relies, as an
intermediate step, in the classification of plane branches modulo unipotent
diffeomorphisms, i.e. germs of biholomorphism $\varphi$ such that the linear
part $D_0 \varphi$ at the origin is a unipotent linear transformation. Since
unipotent diffeomorphisms are always embedded in the one-parameter group of a
formal vector field (cf. Remark \ref{rem:nil-to-unip}), such classes are
complete by Theorem \ref{teo:forcomplete}.  Moreover since any analytic
conjugacy $\varphi$ between curves $\Gamma_1$ and $\Gamma_2$ may be written in
the form $D_0 \varphi \circ \psi$ where
$\psi := (D_0 \varphi)^{-1} \circ \varphi$ has linear part equal to the identity
(see Corollary \ref{cor:equivalence-under-composition-of-flows}), we deduce that
$\Gamma_1$ and $\Gamma_2$ can be connected by two ``segments of geodesic".  More
precisely, there exist germs of singular holomorphic vector fields $X$, $Y$ such
that $(\mathrm{exp}(Y) \circ \mathrm{exp}(X))(\Gamma_1)=\Gamma_2$ (Corollary
\ref{cor:2-geodesic}).  A class of analytic conjugacy ${\mathcal C}$ of a plane
branch $\Gamma$ is identified with the set of left cosets of
$\mathrm{Diff}({\mathbb C}^{2},0)/ \mathrm{Stab} (\Gamma)$ where
$\mathrm{Diff}({\mathbb C}^{2},0)$ is the group of germs of diffeomorphisms
defined in a neighborhood of $0 \in {\mathbb C}^{2}$ and
$\mathrm{Stab} (\Gamma)= \{ \varphi \in \mathrm{Diff}({\mathbb C}^{2},0):
\varphi (\Gamma)=\Gamma\}$ is the stabilizer of $\Gamma$.  It is known that
there exist local biholomorphisms that can not be embedded in the flow of a
formal vector field (see \cite{Zhang-jde-2011} and \cite{ribon-jde-2012}) but to
show that a class is not complete, we need to prove a stronger result, namely
that there exists a left coset $\varphi \circ \mathrm{Stab} (\Gamma)$ in
$\mathrm{Diff}({\mathbb C}^{2},0)/ \mathrm{Stab} (\Gamma)$ such that none of its
elements can be embedded in the flow of a formal vector field.  We will show
that there exist local biholomorphisms $\varphi_0$ such that any
$\varphi \in \mathrm{Diff}({\mathbb C}^{2},0)$ sharing the same second jet as
$\varphi_0$ is not embedded in the flow of a formal vector field.  Then, we
shall prove that there are plane branches $\Gamma$ such that its stabilizer is
small: any element of $\mathrm{Stab} (\Gamma)$ has second jet
equal to the identity map. Combining these two results we obtain that no element
of $\varphi_0 \circ \mathrm{Stab} (\Gamma)$ is embedded in the flow of a formal
vector field. By following the previous ideas we obtain

\begin{proposition}
  \label{pro:noncomplete}
  Let $\Gamma$ be the plane branch with Puiseux parametrization
  $(t^6, t^7 + t^{10} + t^{11})$. Then the class ${\mathcal C}$ of analytic
  conjugacy of $\Gamma$ is non-complete.
\end{proposition}

We can provide a topology in the class ${\mathcal C}$ of a plane branch $\Gamma$
by considering a topology in $\mathrm{Diff}({\mathbb C}^{2},0)$ and the
corresponding quotient topology in the set
$\mathrm{Diff}({\mathbb C}^{2},0)/ \mathrm{Stab} (\Gamma)$.  A natural choice is
the Krull topology (also called ${\mathfrak m}$-adic topology, where
${\mathfrak m}$ is the maximal ideal of ${\mathbb C}[[x,y]]$) where the sets
$S_{k, \varphi}$ of elements of $\mathrm{Diff}({\mathbb C}^{2},0)$ whose $k$-jet
coincides with the $k$-jet of $\varphi$ provide a base of open sets of the
topology by varying $\varphi$ in $\mathrm{Diff}({\mathbb C}^{2},0)$ and $k$ in
${\mathbb N}$.  Proposition \ref{pro:noncomplete} can be reinterpreted as a
genericity property in the class $\mathcal{C}$.

\begin{proposition}
\label{pro:open-krull}
Let ${\mathcal C}$ be the analytic class of the plane branch
$\Gamma$ with Puiseux parametrization $(t^6, t^7 + t^{10} + t^{11})$.  Denote
\begin{equation*}
  {\mathcal C}' = \{ \Gamma' \in {\mathcal C} : \Gamma \ \mathrm{and} \ \Gamma'
  \ \mathrm{are \ connected \ by \ a \ geodesic} \}
\end{equation*}
Then ${\mathcal C} \setminus {\mathcal C}' $ contains an open set of
$\mathcal{C}$ for the Krull topology.  In particular ${\mathcal C}' $ is
not dense in ${\mathcal C}$.
\end{proposition}

The previous result does not hold for other natural topologies.
\begin{proposition}
\label{pro:natural}
Let $\Gamma$ be a plane branch and ${\mathcal C}$ its analytic class of
conjugacy.  Let $\Gamma' \in {\mathcal C}$. Then there exist a holomorphic
deformation $\Gamma_{\epsilon}'$ of $\Gamma'$ by a vector field, defined
in a neighborhood of $\epsilon=0$, $\Gamma_{0}' = \Gamma'$ and a simple
continuous curve $\gamma:[0,1] \to {\mathbb C}$ such that $\gamma (0)=0$ and
$\Gamma$ is connected by a geodesic to $\Gamma_{\gamma(t)}'$ for any
$t \in [0,1]$.
\end{proposition}
Despite the similarities with the tools of reduction of singularities of vector
fields (as in \cite{CanoF3}, \cite{CanoF6}, for instance), our technique is
different: the reduction of singularities seeks a ``simple form" for the
underlying foliation associated to a vector field (and hence, uses techniques
based on invariants like those for foliations as in \cite{Seidenberg1} or
\cite{Cano-Cerveau2}) whereas we are mostly interested in the behaviour of a
vector field under bi-rational maps (blow-ups) and being able to modify it (by
multiplication by a function) in order that the associated flow behaves in a
specific way on an analytic set (the branch). This is, to our knowledge, the
first time this kind of study has been undertaken and we hope to extend it to
other contexts.

Notice that in \cite{Genzmer-2016}, the author provides an algorithm for
computing the dimension of the generic component of the analytic moduli of a
plane branch, using the dual graph of its desingularisation.  Finally, our
techniques are quite different from those of classical deformation theory
\cite{greuel2007introduction}: in this, one is concerned with deformations by
adding a ``small" parameter to the equation of the curve and the aim is to study
the geometric and topological properties of the moduli so obtained. We are
specifically concerned with deformations caused by flows, so that (in a rough
sense) we are adding the parameter at all the orders of the equation.

\section{Notation and Definitions}
Our base ring is $\mathcal{O}=\mathcal{O}_{P}$, the ring of germs of holomorphic
functions in a neighbourhood of a point $P$ of a two-dimensional
complex-analytic manifold, whose base ``set'' we shall usually denote, as is the
custom, $\cc$.  
The maximal ideal of $\mathcal{O}$ will be denoted
$\mathfrak{m}_{0,P}$ or simply $\mathfrak{m}_0$ when no confusion arises.
Assume $P=(0,0) \in {\mathbb C}^{2}$ for simplicity.
We denote $\hat{\mathcal O} ={\mathbb C}[[x,y]]$ and let ${\mathfrak m}$
be the maximal ideal of $\hat{\mathcal O}$.
\begin{definition}
\label{def:krull}
We say that $f,g \in \hat{\mathcal O}$ have the same $k$-jet and we denote
$j^{k} f = j^{k} g$ if $f-g \in {\mathfrak m}^{k+1}$.

Let $(f_k)_{k \geq 1}$ be a sequence in $\hat{\mathcal O}$. Then it converges to
$f \in \hat{\mathcal O}$ in the ${\mathfrak m}$-adic topology (or also the Krull
topology) if for any $l \geq 1$ there exists $k_0 \geq 1$ such that
$j^{l} f_k = j^{l} f$ for any $k \geq k_0$.
\end{definition}
 \begin{definition}
 We say that $X$ is a {\it vector field} if is a
  $\mathbb{C}-$derivation
$X:\mathcal{O}_P\rightarrow \mathcal{O}_P$ continuous for the
$\mathfrak{m}_P-$adic topology. It is {\it singular}
\footnote{As a matter of fact,
the expression should be ``$P$ is an equilibrium point of $X$'' but we are
indulging the custom.}
if $X({\mathfrak m}_P) = {\mathfrak m}_P$ and {\it regular} otherwise.
In the case $P=(0,0) \in {\mathbb C}^{2}$ we write
  \begin{equation*}
 X = A(x,y) \frac{\partial}{\partial x} + B(x,y) \frac{\partial}{\partial y}.
  \end{equation*}
  where $A := X(x)$ and $B:=X(y)$ belong to ${\mathcal O}$.
  Analogously by replacing ${\mathcal O}$, ${\mathfrak m}_0$ with
  $\hat{\mathcal O}$, ${\mathfrak m}$, we can define formal vector fields.
 \end{definition}
 \begin{definition}
 Let $X$ be a formal singular vector field. We say that $X$ is {\it nilpotent} if its linear part is
 a nilpotent vector field. 
 \end{definition}

 We shall also say that $P$ is a singular
point for $X$ (especially, but not only, when $X$ can be understood as a vector
field on a larger analytic manifold). Finally, $X$ is \emph{\tsing{}} at $P$ (or
$P$ is a \emph{true singularity} of $X$) if it is singular and there do not
exist a regular vector field $Y$ and a regular holomorphic function $f\in
\mathcal{O}_P$ such that $X = f^mY$ for some positive integer $m$ (this is
related to what is called a \emph{strictly singular} point in
\cite{Brochero-Cano-Lopez}). Note that all these definitions are given for the
local case: we shall be explicit when dealing with non-local situations.

The \emph{multiplicity} of a  formal  vector field $X$ is the largest non-negative
integer $m$ such that $X(\mathfrak{m})\subset \mathfrak{m}^{m}$. Thus, a
non-singular vector field has multiplicity $0$ and, in general, if
$X=a(x,y)\pt{x} + b(x,y)\pt{y}$, then the multiplicity of $X$ is the smallest of
the multiplicities of $a(x,y)$ and $b(x,y)$.

An analytic \emph{branch} (simply branch) at $P$ is any reduced and irreducible
curve $\Gamma\subset \cc$. Unless otherwise specified, all our curves will be
analytic branches and they will be defined either by a reduced and irreducible
holomorphic function $f\in \mathfrak{m}_P$ or by a Puiseux expansion
$\varphi(t)=(x(t),y(t))$ when local coordinates at $P$ are already chosen. All the
results related to desingularisation of plane branches (and, as a requirement,
finite sequences of point blow-ups, exceptional divisors, etc.) and their
topological (not analytic) structure are assumed known: two good modern
references are \cite{Casas} and \cite{Wall}.

Consider a point $P$ belonging to a two-dimensional complex analytic manifold
$\mathcal{M}$. Denote by $\mathcal{M}_P$ the germ of $\mathcal{M}$ at $P$ (which
is, essentially, the same thing as $\cc$). As our work is based on the process
of point blow-ups, we need the following
\begin{definition}\label{def:pull-back-of-vector-field} Let $X$ be a singular
vector field at $P$ and let $\pi:\mathcal{X}\rightarrow \mathcal{M}_P$ be the
blow-up with centre $P$. The unique holomorphic vector field $\overline{X}$
on the whole $\mathcal{X}$ such that $\pi_{\ast}(\overline{X}) = X$ outside of
the exceptional divisor $\pi^{-1}(P)$ is called the \emph{pull-back of $X$ to
$\mathcal{X}$}.
\end{definition} The fact that $\overline{X}$ exists is due to the singularity
of $X$ at $P$: otherwise, $\overline{X}$ is not defined (it has ``poles'' on the
exceptional divisor).

\begin{remark*}\label{rem:we-take-true-pull-back} Notice that we are taking the
``true'' pull-back of $X$ on $\mathcal{X}$: we are interested in the
\emph{dynamics} of $X$, not just in the geometric structure of its integral
curves. Thus, if $(x,y)$ are local coordinates at $P$ and one looks at the chart
of $\pi$ with equations $x=\overline{x},y=\overline{x}\overline{y}$ and
  \begin{equation*}
    X = a(x,y) \pt{x} + b(x,y) \pt{y}
  \end{equation*}
  for some $a(x,y),b(x,y)\in \mathfrak{m}_P$, then on the chart
  $(\overline{x}, \overline{y})$, the local equation of $\overline{X}$ is given
  by
  \begin{equation*}
    \overline{X} = a(\overline{x}, \overline{x}\overline{y}) \pt{\overline{x}} +
    \frac{1}{\overline{x}}(
      -\overline{y}a(\overline{x},\overline{x}\overline{y}) +
      b(\overline{x}, \overline{x}\overline{y}))\pt{\overline{y}},
  \end{equation*}
  expression which shows why $X$ must have a singularity at $P$ in order to
  admit a pull-back to $\mathcal{X}$. As the reader will have noticed, we
  \emph{do not eliminate} the possible common factor $\overline{x}$ in the
  expression of $\overline{X}$. This implies that, usually, the pull-back of a
  singular vector field will not be \tsing{}: it will have some true
  singularities on the exceptional divisor but most of the points will be just
  equilibrium points such that, near them, $\overline{X}$ is of the form
  $\overline{x}^mY$ for some non-negative integer $m$ and non-singular vector
  field $Y$.

  The reader familiar with the theory of plane holomorphic foliations will
notice the similarity and the differences between our approach and the one
common in those works. This difference is exactly what makes our technique
useful for studying deformations.
\end{remark*}

Anyway, we can consider the desingularisation of the underlying foliation of a
singular vector field. The following result is a restatement of the main
one in \cite{Seidenberg1}.
\begin{theorem}[cf. \cite{Seidenberg1}]\label{the:reduction-of-singularities}
  Let  $X$ be a singular vector field at $P\in \mathcal{M}_P$. There
  is a finite sequence of blow-ups $\pi:\mathcal{X}\rightarrow \mathcal{M}_P$:
  \begin{equation*}
    \mathcal{X}=\mathcal{X}_{N}\xrightarrow{\pi_{N-1}}
    \mathcal{X}_{N-1}\xrightarrow{\pi_{N-2}}\cdots
    \xrightarrow{\pi_1}\mathcal{X}_{1}
    \xrightarrow{\pi_0}\cc
  \end{equation*}
  $\pi=\pi_0\circ\cdots\circ\pi_{N-1}$ whose centres $(P_i)_{i=0}^{N-1}$ are
  singular points for the respective pull-back of $X$ and such that the
  pull-back $\overline{X}$ of $X$ on $\mathcal{X}$ has a finite number of true
  singularities and at any of these, say $Q$, $\overline{X}$ admits an
  expression of the form
  \begin{equation*}
    x^ay^b\left(\mu x \pt{x} + \lambda y \pt{y} + h.o.t.\right)
  \end{equation*}
  where $(x,y)$ are local coordinates at $Q$, the exceptional divisor is
included in $xy=0$, $\mu\neq 0$ and $\lambda/\mu \not\in \mathbb{Q}_{>0}$. The
shortest non-empty sequence of blow-ups for which this happens is called
\textup{the minimal reduction of singularities of $X$}.
\end{theorem}

Let $\Gamma$ and $X$ be an analytic branch and a singular vector field at $\cc$.
Let $\pi_i:\mathcal{X}_{i+1} \rightarrow \mathcal{X}_{i}$ be the infinite
sequence of blow-ups with centre $P_{i}$, the intersection of the strict
transform $\overline{\Gamma}_{i}$ of $\Gamma$ with the corresponding exceptional
divisor (with $\mathcal{X}_0=\cc$ and $\overline{\Gamma}_0 = \Gamma$). 
The next result follows easily from
the fact that $\Gamma$ is analytic:
\begin{proposition}\label{pro:non-branch-separates}
  With the notation of the last paragraph, $\Gamma$ is invariant by $X$ if and
  only if $P_i$ is a singular point of the pull-back $\overline{X}_{i}$ of $X$
  to $\mathcal{X}_{i}$ for any $i \geq 0$.  In particular, if $\Gamma$ is not
  invariant by $X$ then there exists $i_{0} \geq 0$ such that $P_i$ is a
  singular point of $\overline{X}_{i}$ of $X$ for any $0 \leq i \leq i_{0}$ but
  $P_{i_{0} +1}$ is a regular point of $\overline{X}_{i_{0}+1}$.
\end{proposition} 

This result provides the following
\begin{definition}\label{def:shared-path}
  The \emph{path shared} by a non-invariant analytic branch $\Gamma$ and a
singular vector field $X$ is the sequence $(P_0,P_1,\dots, P_{i_0+1})$ given by
Proposition \ref{pro:non-branch-separates}. Notice that we include in the shared
path the point at which the pull-back of $X$ is non-singular.
\end{definition}

\begin{remark*}
  The last point shared by $X$ and $\Gamma$ could be a singular
point of the strict transform of $\Gamma$: we only require it to be
a regular point for the pull-back of $X$.
\end{remark*}

The following result will be important in the study of the relation between a
curve and its deformation:
\begin{lemma}\label{lem:shared-path-does-not-end-in-corner}
  Let $(P_i)_{i=0}^N$ be the shared path between $\Gamma$ and a singular
  vector field $X$. The last point $P_N$ is not a corner of the exceptional
  divisor.
\end{lemma}
\begin{proof}
  This is because after blowing up a singular point, the exceptional divisor is
  always invariant for the pull-back. If $P_N$ were a corner, then the pull-back
  $\overline{X}$ at $P_N$ would possess at least two invariant curves: both
  components of the exceptional divisor. This would imply that $P_N$ is singular
  for $\overline{X}$, which contradicts the definition.
\end{proof}

We introduce now our main object of study:
\begin{definition}\label{def:deformation}
  Given a singular germ of analytic vector field $X$ and an irreducible germ of
  analytic plane curve $\Gamma$ at $\cc$ with $\Gamma\equiv (f=0)$ for
  $f\in \mathcal{O}$, we define the \emph{deformation of $\Gamma$ caused by $X$
    or by the flow associated to $X$} as the family
\begin{equation*}
  \Gamma_{\epsilon} \equiv \left(\sum_{n=0}^{\infty}\frac{\epsilon^n}{n!}X^n(f)
    = 0\right)
  \equiv
  \left( f + \sum_{n=1}^{\infty}\frac{\epsilon^n}{n!}X^n(f) = 0 \right).
\end{equation*}
We shall refer either to the whole family or to any of its elements as ``the
deformation of $\Gamma$''.
\end{definition}
Notice that, because $X$ is singular, if its multiplicity is greater than $1$,
then the local equation of $\Gamma_{\epsilon}$ is, roughly speaking, a higher
order deformation of the local equation of $\Gamma$, in the sense that the terms
added to $f$ are of order at least one more than the vanishing order of $f$. 
In any case, it is clear that
the deformation of a non-singular analytic branch by a singular vector field is
non-singular for $\epsilon$ small enough.

The following consequence of the formula for the higher derivative of a product
is what makes blow-ups a sensible tool for studying deformations caused by
vector fields:
\begin{lemma}\label{lem:deformation-of-strict-transform}
  Let $X$ be a singular vector field at $\cc$, $\pi:\mathcal{X}\rightarrow \cc$
  be the blow-up with centre $(0,0)$ and $\overline{X}$ the pull-back of $X$
  by $\pi$. If $\Gamma\equiv(f=0)$ is an analytic branch through $(0,0)$ and
  $\overline{\Gamma}$ is its strict transform by $\pi$, then
  \begin{equation*}
    \overline{\Gamma_{\epsilon}} = \overline{\Gamma}_{\epsilon},
  \end{equation*}
  that is: the strict transform of the deformation of $\Gamma$ by $X$ is the
  deformation of the strict transform $\overline{\Gamma}$ by
  $\overline{X}$. This generalises to any finite sequence of blow-ups with
  centres singular points of $X$ and its successive pull-backs.
\end{lemma}

Finally, the one-parameter group of diffeomorphisms of a vector field and the
one of its pull-back are essentially the same object:
\begin{lemma}\label{lem:pull-back-of-diffeos}
  Let $X$ be a singular vector field at $\cc$ and $\{\psi_{X,s}(z)\}_{s}$ its
  one-parameter group of germs of diffeomorphisms. Let
  $\pi:\mathcal{X}\rightarrow\cc$ be a sequence of blow-ups whose centres are
  singular points of each pull-back of $X$ and let $\overline{X}$ be the
  pull-back of $X$ to $\mathcal{X}$. The diffeomorphism associated to
  $\overline{X}$ for the value $s$ of the parameter is the unique holomorphic
  extension $\overline{\psi}_{\overline{X},s}$ to the whole $\mathcal{X}$ of the
  diffeomorphism $\pi^{-1}\circ\psi_{X,s}\circ{\pi}$ defined on
  $\mathcal{X}\setminus \pi^{-1}(0,0)$.
\end{lemma}

\section{Main results}
The  deformation of an analytic branch $\Gamma$ caused by a singular
vector field $X$ has a nice behaviour due to Cauchy-Kowalewski's Theorem:
\begin{proposition}\label{pro:deformations-share}
  Let $X$ be a singular analytic vector field at $\cc$ and let $\Gamma$ be an
  analytic plane branch which is not invariant for $X$. Then $\ce$ and $\Gamma$
  are analytically conjugated
  and they share the same path with $X$
  except possibly the last point: for $\epsilon$ small enough, the last shared
  point is certainly different.
\end{proposition}
\begin{proof}
  Let $(P_i)_{i=0}^{N}$ be the path shared by $X$ and $\Gamma$. By Lemma
  \ref{lem:shared-path-does-not-end-in-corner}, $P_N$ is not a corner of the
  exceptional divisor. By definition, the vector field $\overline{X}$ is
  non-singular at $P_N$ and it is tangent to the exceptional divisor
  $E=\pi^{-1}(0,0)$. This implies that $\ce$ meets the exceptional divisor away
  from $P_N$ for $\epsilon$ small enough.
\end{proof}


For the sake of clarity let us recall the definition of intersection
multiplicity.
\begin{definition}\label{def:intersection-number}
  Let $\Delta\equiv(g(x,y)=0)$ be an analytic curve in $\cc$ which does not
  contain $\Gamma$. The \emph{intersection multiplicity}
   ($\Gamma \cap \Delta)_{(0,0)}$ (also denoted by $(f,g)_{(0,0)}$) 
  of $\Gamma$ and $\Delta$ at $(0,0)$ is the (finite) number
  \begin{equation*}
    (\Gamma \cap \Delta)_{(0,0)} = \dim_{\mathbb{C}}\mathbb{C}\{x,y\}/(f,g).
  \end{equation*}
  In the case we are dealing with, where  $\Gamma$ is a branch,
  this number can be computed as
  \begin{equation*}
    (\Gamma \cap \Delta)_{(0,0)} = \ord_t(g(\varphi(t))
  \end{equation*}
  where $\varphi(t)$ is any irreducible Puiseux parametrization of $\Gamma$. The
  sub-index $(0,0)$ is usually omitted.

\end{definition}

A direct consequence of Proposition \ref{pro:deformations-share} is
\begin{corollary}
  \label{cor:intersection-gamma-and-gammae}
  Let $\Gamma\equiv(f=0)$ be a (possibly singular) analytic branch at $\cc$ 
  that is not invariant by 
  a singular analytic vector field $X$. If $n_0, n_1, \dots, n_N$ is the
  sequence of multiplicities of $\Gamma$ at the points of the path it shares
  with $X$, then the intersection multiplicity of $\Gamma$ and $\ce$ is given
  by:
  \begin{equation*}
    (\Gamma, \ce)_{(0,0)} = \sum_{i=0}^{N-1}  n_i^2 =
    \min \{ (X^{n}(f), f)_{(0,0)} : n \geq 1 \}
  \end{equation*}
  for $0<\abs{\epsilon}\ll 1$.
\end{corollary}
\begin{proof}
  As  the sequence of infinitely near points
  shared by $\Gamma$ and $\ce$ is the shared path between $X$ and $\Gamma$
  except the last point (for $\epsilon\ll 1$), Noether's formula (see, for
  example \cite{Casas}) gives
  \begin{equation*}
    (\Gamma, \ce)_{0,0} = \sum_{i=0}^{N-1} n_i n_{P_i}  (\ce)
  \end{equation*}
  where $n_{P_i}(\ce)$ denotes the multiplicity of the strict transform of $\ce$
  at $P_i$. Since $\Gamma$ and $\ce$ are topologically equivalent (as they are
  analytically conjugated), their sequence of multiplicities at their infinitely
  near points are the same: $n_{P_i}(\ce)=n_i$  and the first equality
  follows. Notice that we need set $\epsilon\ll 1$ because $\Gamma$ and $\ce$
  might share more points for $\epsilon$ not small enough. For the second
  equality, let $k = \min \{ (X^{n}(f), f)_{(0,0)} : n \geq 1 \}$. Certainly,
  $k\leq (\Gamma, \Gamma_{\epsilon})_{(0,0)}$. Notice that
  \begin{equation*}
    (\Gamma, \Gamma_{\epsilon})_{(0,0)} =
    \ord_{t} \left(\sum_{r\in T_k} \frac{\epsilon^r}{r!}
    a_r t^k + \hot\right)
  \end{equation*}
  where $T_k=\{r \in \mathbb{N} : (f,X^r(f))_{(0,0)}=k \} $ and $a_r$ is the
  term of order $k$ in $X^r(f)(\varphi(t))$, for a parametrization $\varphi(t)$
  of $\Gamma \equiv (f=0)$. If $(\Gamma,\Gamma_{\epsilon})_{(0,0)}$ were
  strictly greater than $k$ for some $\epsilon \neq 0$ in every pointed 
  neighborhood of $0$, then
  \begin{equation*}
    \sum_{r\in T_k}\frac{a_r}{r!}\epsilon^{r}=0
  \end{equation*}
  for all $\epsilon \in {\mathbb C}$ by the isolated zeros principle, so that
  $a_r=0$ for all $r$, against the assumption.
\end{proof}
\begin{definition}
  \label{def:tangency-order-brunella}
  The \emph{tangency order} between $X$ and $\Gamma$ is defined as
  $\tang_{(0,0)}(X, \Gamma) =(X(f),f)_{(0,0)}$ (see
  \cite{brunella2015birational}) .
\end{definition}  
\begin{lemma}
  We have
  \begin{equation*}
    \tang_{(0,0)}(X, \Gamma) = (\Gamma, \ce)
  \end{equation*}
  for $0<\epsilon\ll 1$ when $\Gamma$ is non-singular.
\end{lemma}  
\begin{proof}
  We can assume that $\Gamma$ is not invariant by $X$, since otherwise 
  $\tang_{(0,0)}(X, \Gamma) = \infty$ and $\Gamma=\Gamma_{\epsilon}$ for any 
  $\epsilon \in {\mathbb C}$.
  As $\Gamma$ is non-singular and it is not invariant for $X$, after a change of
  coordinates, we may assume $f=y$ and $X(y)\not\in(y)$. Writing
  \begin{equation*}
    X = A(x,y)\frac{\partial }{\partial x} + B(x,y)\frac{\partial }{\partial y}
  \end{equation*}
  we obtain $B(x,y)=a(x^k + \hot) + y(\overline{B}(x,y))$ for some $k>0$ and
  $a\neq 0$. As $A(0,0)\neq 0$, an easy inductive argument implies that
  \begin{equation*}
    \ord_x(X^k(y)(x,0)) \geq k
  \end{equation*}
  which is what we need, by Corollary \ref{cor:intersection-gamma-and-gammae}.
\end{proof}
\subsection{Vector fields, differential forms and curves}
Consider now a branch $\Gamma$ which, for the sake of simplicity, we assume
tangent to the $OX$ axis, so $\Gamma\equiv (f(x,y)=0)$ with $f(x,y)=y^n +
\hot$ It is well known that $\Gamma$ admits what is called an \emph{irreducible
  Puiseux parametrization}
\begin{equation}\label{eq:irred-puiseux}
  \varphi(t) \equiv (x(t),y(t)) = \left( t^n, \sum\limits_{i\geq n} a_it^i \right)
\end{equation}
where $\varphi(t)$ is not of the form $\varphi(t^k)$ for any $k\geq 2$; the
greatest common divisor of $n$ and the exponents appearing in $y(t)$ is $1$. Up
to replacing $\Gamma$ with its conjugate by some local diffeomorphism of the
form $(x,y) \mapsto (x, y+ a(x))$, one easily deduces that there exists what we
shall call a \emph{prepared Puiseux parametrization}:
\begin{definition}\label{def:prepared-puiseux-expansion}
  A \emph{prepared Puiseux parametrization} of $\Gamma$ is an irreducible
  Puiseux parametrization such that $m>n$ and $n\nmid m$.
\end{definition}
Before proceeding any further, let us recall some definitions:

\begin{definition}\label{def:semigroup and conductor}
  The \emph{semigroup} $S_{\Gamma}$ (or simply $S$) associated to $\Gamma$ is
  the set
  \begin{equation*}
    S_{\Gamma} = \left\{ (\Gamma\cap \Delta)_{(0,0)} \,:\,
      \Delta\equiv(f(x,y)=0), f(x,y)\in \mathbb{C}\{x,y\},
      \Gamma\not\subset \Delta \right\}.
  \end{equation*}
  It is a sub-semigroup of $\mathbb{N}$. It is well known (due to the fact that
  $\Gamma$ is a branch) that there is $c\in S_{\Gamma}$ such that
  $p\geq c$ implies $p\in S_{\Gamma}$.
  The least $c$ satisfying this property is called
  the \emph{conductor} of $\Gamma$.
\end{definition}

Given a differential form $\omega\in \Omega^1_{\mathcal{O}}$, say
$\omega=a(x,y)dx + b(x,y)dy$, the \emph{contact} of $\omega$ with $\Gamma$ is
defined (as in \cite{Zariski4}) as
\begin{equation*}
  \upsilon_{\Gamma}(\omega) = \ord_t\big(a(x(t),y(t))\dot{x}(t) +
  b(x(t),y(t))\dot{y}(t)\big) + 1,
\end{equation*}
which does not depend on the parametrization of $\Gamma$. On the other hand,
given a vector field $X$, say $X=A(x,y)\pt{x} + B(x,y)\pt{y}$,
let us calculate $\mathrm{tang}_{(0,0)} (X, \Gamma)$.
We have $\ptt{f}{y}\neq 0$ and (certainly) $\dot{x}(t)\neq 0$.
Since $f(x(t),y(t))=0$, we deduce
\begin{equation}\label{eq:partial-f-wrt-x}
  \ptt{f}{x}\dot{x}(t) + \ptt{f}{y}\dot{y}(t) = 0,
\end{equation}
which can be rewritten as
\begin{equation*}
  \ptt{f}{x} = -\frac{\dot{y}(t)}{\dot{x}(t)}\ptt{f}{y},
\end{equation*}
so that, when computing the tangency order $\mathrm{tang}_{(0,0)} (X, \Gamma)$,
one gets
\begin{equation*}
  X(f)(x(t),y(t)) = A(x(t),y(t))\ptt{f}{x}(x(t),y(t)) +
  B(x(t),y(t))\ptt{f}{y}(x(t),y(t)),
\end{equation*}
which substituting \eqref{eq:partial-f-wrt-x}, gives
\begin{equation*}
  X(f)(x(t),y(t))\dot{x}(t) = \ptt{f}{y}(x(t),y(t))
  \big( -A(x(t),y(t))\dot{y}(t) + B(x(t),y(t))\dot{x}(t) \big),
\end{equation*}
that leads to the following valuative formula:
\begin{equation*}
  \ord_t\big(X(f)(x(t),y(t))\big) + \ord_t(\dot{x}(t)) =
  \ord_t \left( \ptt{f}{y}(x(t),y(t)) \right) +
  \upsilon_{\Gamma} (\overline{\omega}) - 1,
\end{equation*}
for $\overline{\omega} = B(x,y)dx - A(x,y)dy$. It is well-known (see, for
example \cite{Zariski4}) that
\begin{equation*}
  \ord_{t} \left( \ptt{f}{y}(x(t),y(t)) \right) = c + n - 1
\end{equation*}
where $c$ is the conductor of $S_{\Gamma}$. Hence, we get
\begin{equation*}
  \mathrm{tang}_{(0,0)} (X, \Gamma)+ (n-1) = c + (n-1)
  + \upsilon_{\Gamma}(\overline{\omega}) - 1,
\end{equation*}
that is:
\begin{equation*}
   \mathrm{tang}_{(0,0)} (X, \Gamma)=
  \upsilon_{\Gamma}(\overline{\omega}) + c - 1.
\end{equation*}
Thus, we might define
$\upsilon_{\Gamma}(X) \coloneqq   \mathrm{tang}_{(0,0)} (X, \Gamma)- c + 1$ and
obtain, in a natural way:
\begin{equation*}
  \upsilon_{\Gamma}(X) = \upsilon_{\Gamma}(\overline{\omega}).
\end{equation*}
Dually, we obtain the following formula for the conductor:
\begin{corollary}\label{cor:conductor-from-vector-and-form}
  Let $\Gamma\equiv(f=0)$ be a singular branch at $\cc$ and
  $X=A(x,y)\pt{x} + B(x,y)\pt{y}$ any singular vector field. If
  $\omega = -B(x,y)dx+A(x,y)dy$  then 
  \begin{equation*}
    c =   \mathrm{tang}_{(0,0)} (X, \Gamma) - \upsilon_{\Gamma}(w) + 1.
  \end{equation*}
\end{corollary}

The next result follows from a simple (classical) computation:
\begin{lemma}\label{lem:upsilon-omega-and-upsilon-transform}
  Let $\omega$ be a singular differential form in $\cc$ and
  $\pi:\mathcal{X}\rightarrow \cc$ be the blow-up of $\cc$ with centre
  $(0,0)$ with equations $x=\overline{x}, y=\overline{x}\overline{y}$. Let
  $\Gamma$ be a branch (singular or not) at $\cc$
  whose tangent cone is not $x=0$. Consider the differential
  form in $\mathcal{X}$ given by
  $\overline{\omega} = (\pi^{\ast}\omega)/\overline{x}$ (which is the dual form
  of the pull-back of $X$ to $\mathcal{X}$) and the strict transform
  $\overline{\Gamma}$ of $\Gamma$, whose intersection with $\pi^{-1}(0,0)$ is
  $P$. If $n$ is the multiplicity of $\Gamma$ at $(0,0)$, then
  \begin{equation*}
    \upsilon_{\Gamma}(\omega) =
    \upsilon_{\overline{\Gamma}}(\overline{\omega}) + n.
  \end{equation*}
\end{lemma} 
\begin{corollary} \label{cor:kahlercontact-exponent-and-blow-up} 
  Let $\Gamma$ be an analytic branch at $\cc$ 
  that is not invariant by 
  a singular analytic vector field $X=A(x,y)\pt{x} + B(x,y)\pt{y}$.  
  Let $\omega = -B(x,y)dx+A(x,y)dy$ be the ``dual'' differential form
  of $X$. Then
   \begin{equation*}
    \upsilon_{\Gamma}(\omega) = n_{N-1} + \sum_{j=0}^{N-1}n_j
  \end{equation*}
  where $n_0, n_1, \dots, n_N$ is the
  sequence of multiplicities of $\Gamma$ at the points of the path it shares
  with $X$.  It depends only on $\Gamma$ and $N$.
\end{corollary}
\begin{proof}
 Let $(P_j)_{j=0}^{N}$ be the path shared by $X$ and $\Gamma$ and $n_j$ the
  multiplicity of the strict transform $\overline{\Gamma}_j$ of $\Gamma$ at
  $P_j$. For each $j=0,\dots, N$, if we denote by $\overline{X}_j$ the pull-back
  of $X$ to the respective space (so that $\overline{X}_0 = X$) and
  $\overline{\omega}_j$ its dual form. We have
  \begin{equation*}
    \upsilon_{\overline{\Gamma}_j}(\overline{\omega}_{j}) =
    \upsilon_{\overline{\Gamma}_{j-1}}(\overline{\omega}_{j-1}) - n_{j-1},
    \mbox{ for }j=1,\dots,{N-1},
  \end{equation*}
  by Lemma \ref{lem:upsilon-omega-and-upsilon-transform}.
  As $\overline{X}_{N-1}$ does not preserve the tangent cone of $\overline{\Gamma}_{N-1}$
  (because $P_N$ is not a singular point of $\overline{X}_N$), we may assume
  that $\overline{\Gamma}_{N-1}=(t^{n_{N-1}}, t^q+\hot)$ with $q>n_{N-1}$.
  The form $\overline{\omega}_{N-1}$ can be written 
   \begin{equation*}
    \overline{\omega}_{N-1} = (ax+by+\hot)dx + (cx+dy+\hot)dy
  \end{equation*}
  with $a\neq 0$, which gives
  $\upsilon_{\Gamma_{N-1}}(\overline{\omega}_{N-1})=2n_{N-1}$ and, from Lemma
  \ref{lem:upsilon-omega-and-upsilon-transform}, we get
  \begin{equation*}
    \upsilon_{\overline{\Gamma}_N}(\overline{\omega}_N) = n_{N-1}
  \end{equation*}
  and then
  \begin{equation*}
    n_{N-1} = \upsilon_{\overline{\Gamma}_N}(\overline{\omega}_N) =
    \upsilon_{\Gamma}(\omega)-n_0-n_1-\cdots - n_{N-1}
  \end{equation*}
  and the result follows.
\end{proof}

\subsection{The shared path and Puiseux's expansion}
The concept of the path shared by a singular vector field and an analytic branch
is deeply related to the Puiseux expansion of the branch and the contact between
the branch and the vector field (or the branch and its deformation).

Start with an analytic branch $\Gamma$ at $\cc$ which is not tangent to the $OY$
axis, so that it admits a Puiseux expansion\footnote{We always assume the
  parametrizations to be irreducible.} of the form
\begin{equation*}
  \Gamma\equiv \varphi(t) = (t^n, \sum\limits_{i\geq n} a_it^i),
\end{equation*}
where $\Gamma$ is regular if and only if $n=1$.
Let $X$ be a singular vector field at
$\cc$. We need a technical result foremost:

\begin{lemma}\label{lem:root-of-puiseux-series}
  Let $a(z)=\sum_{i\geq n} a_i(z)t^i$ be a power series with integer exponents
  such that $n\geq 1$ and each $a_i(z)$ is a holomorphic function in $z$ (each
  with its own radius of convergence), with $a_n(0)\neq 0$. If
  $r(z)=\sum_{i\geq 1} r_i(z)t^i$ is such that
  \begin{equation*}
    r(z)^n = a(z),
  \end{equation*}
  then $r_i(z)$ are also holomorphic functions in $z$ and $r_1(0)\neq 0$.
\end{lemma}
\begin{proof}
  The proof is done by induction on $i$ or, what amounts to the same, by the
  method of indeterminate coefficients. Actually, one can prove that there exist
  polynomials $P_j(z)$ in $j-1$ variables such that
  \begin{equation*}
    r(z)^n =
    \bigg( r_1(z)^nt^n +  \sum_{j=2}^{\infty} \big(nr_1(z)^{n-1}r_j(z) +
    P_j(r_1(z),\dots,r_{j-1}(z))\big)t^{n+j-1} \bigg)
  \end{equation*}
  from which the result follows.
\end{proof}

\begin{definition}
Assume that $X$ is a singular vector field. Let $P$ and $Q$ be the points defined in the
divisor of the blow-up of the origin by the tangent cone of $\Gamma$ and $x=0$ respectively.
We say that $X$ is {\it prepared} relatively to
$\Gamma$ if either $P$ or $Q$ is a singular point of $\overline{X}$.
\end{definition}
\begin{remark}
  Given a vector field $X$ we may assume that it is prepared up to a linear
  change of coordinates that preserves the tangent cone of $\Gamma$ at $(0,0)$,
  which is assumed to be $y=0$.  The preparation guarantees that no curve of the
  form $\Gamma_{\epsilon}$ has the $OY$ axis as its tangent cone.  Given a
  vector field $X = a(x,y) \pt{x} + b(x,y) \pt{y}$, it is prepared relatively to
  $\Gamma$ if and only if $\frac{\partial a}{\partial y} (0,0) =0$ or
  $\frac{\partial b}{\partial x} (0,0) =0$.  The transform $\overline{X}$ of a
  singular vector field $X$ is of the form
  \begin{equation*}
    \overline{X} = xA(x,y)\pt{x} +
    (\mu_0 + \mu_{1} y + \mu_2 y^2 + xB(x,y))\pt{y}.
 \end{equation*}
 and $X$ is prepared if and only if $\mu_0 =0$ or $\mu_2=0$.
\end{remark}

Consider the deformation $\Gamma_{\epsilon}$ of $\Gamma$ by $X$.

\begin{proposition}\label{pro:curve-and-puiseux-deformation-and-field}
  Assume that $X$ is prepared relatively to $\Gamma$.
   The deformation $\Gamma_{\epsilon}$ admits an
  irreducible Puiseux parametrization:
  \begin{equation*}
    \Gamma_{\epsilon} \equiv \varphi_{\epsilon} (t)=
    \bigg(t^n,\,  \sum_{\substack{i\geq n}}\tilde{a}_{i}(\epsilon)t^i\bigg).
  \end{equation*}
  with $\tilde{a}_i(\epsilon)$ being an entire function 
  in $\epsilon$ and
  $\tilde{a}_i(0)=a_i$ for all $i\geq n$.
\end{proposition}
\begin{proof} 
  Let ${ \{ \psi_s \}}_{s \in {\mathbb C}}$ the one-parameter group associated to $X$.
  Let $(\gamma_1 (t), \gamma_2 (t))= (t^n, \sum\limits_{i\geq n} a_it^i)$ be a Puiseux
  parametrization.
  We define the  map
  \begin{equation*}
 (\gamma_1(s,t), \gamma_2(s,t)) = \psi_{s} (\gamma_1 (t), \gamma_2(t)).
  \end{equation*}
  It is well-defined and holomorphic in a neighborhood of $t=0$.  As a
  consequence $\gamma_1 (s,t)$ and $\gamma_2 (s,t)$ are of the form
  $\gamma_1 (s,t) = \sum_{j=1}^{\infty} b_j (s) t^{j}$ and
  $\gamma_2 (s,t) = \sum_{j=1}^{\infty} c_j (s) t^{j}$ where $b_j$ and $c_j$ are
  entire functions for any $j \geq 1$.  Since the multiplicity of every curve
  $\Gamma_s$ is equal to $n$ for $s \in {\mathbb C}$, all coefficients $b_j$ and
  $c_j$ with $j < n$ are identically $0$.  Moreover $b_n$ is a never vanishing
  entire function, otherwise the tangent cone of $\Gamma_s$ would be $x=0$ for
  some $s \in {\mathbb C}$. Lemma \ref{lem:root-of-puiseux-series} implies that
  there exists $\beta (s,t) = \sum_{j=1}^{\infty} \tilde{b}_j (s) t^{j}$ such
  that $\beta (s,t)^{n} = \gamma_1 (s,t)$ where $\tilde{b}_j$ is an entire
  function for any $j \geq 1$ and $\tilde{b}_1$ is never vanishing.  Denote by
  $(s, \alpha (s,t))$ the inverse map of $(s, \beta (s,t))$. It is well-defined
  in a neighborhood of $t$ since $\tilde{b}_1$ never vanishes. The
  parametrization that we are looking for is
  $(t^{n}, \gamma_2 (s, \alpha (s,t)))$. 
\end{proof}
As a corollary we obtain  the following result. 

\begin{corollary}\label{cor:curve-and-puiseux-deformation-and-field}
  For $\epsilon$ small enough, the deformation $\Gamma_{\epsilon}$ admits an
  irreducible Puiseux parametrization:
  \begin{equation*}
    \Gamma_{\epsilon} \equiv \varphi_{\epsilon} (t)=
    \bigg(t^n,\,  \sum_{\substack{i\geq n}}\tilde{a}_{i}(\epsilon)t^i\bigg).
  \end{equation*}
  with $\tilde{a}_i(\epsilon)$ holomorphic
  in $\epsilon$ and
  $\tilde{a}_i(0)=a_i$ for all $i\geq n$.
\end{corollary}
The algebraic counterpart to the geometric concept of the shared path is the \emph{contact exponent}:
\begin{definition}\label{def:contact-exponent}
  The \emph{contact exponent} of a singular vector field $X$ with an analytic
  branch $\Gamma$ at a point $P$ of a complex analytic surface $\mathcal{M}$,
  denoted $(X,\Gamma)_{P}$, is the least $i$ such that $\tilde{a}_i(\epsilon)$
  is not constant in Proposition
  \ref{pro:curve-and-puiseux-deformation-and-field} (for any irreducible Puiseux
  parametrization of $\Gamma$).
\end{definition}
\begin{remark}
  The contact exponent is independent of the choice of coordinates.  Consider a
  local biholomorphism $\phi \in \mathrm{Diff}({\mathbb C}^{2},0)$ such that the
  linear part $D_0 \phi$ at the origin does not send the tangent line to
  $\Gamma$ at $0$ to the $OY$ axis. It can be shown that
  $(X,\Gamma)_{(0,0)} = (\phi_{*} X, \phi (\Gamma))_{(0,0)}$ by a simple
  calculation.
\end{remark}
\begin{remark}\label{rem:contact-less-than-m-is-multiple-of-n}
  Assume $\Gamma$ has an irreducible Puiseux expansion $(t^n, at^m +\hot)$ with
  $a\neq 0$ and $n<m$. If
  $j=(X,\Gamma)_{(0,0)}<m$  then $j$ must be a multiple of $n$:
  otherwise the topological types of $\Gamma$ and $\Gamma_{\epsilon}$ would be
  different, which is impossible because they are analytically equivalent.
\end{remark}

One has a formula analogue to that of Lemma
\ref{lem:upsilon-omega-and-upsilon-transform}, which provides the
relation between the shared path and the contact exponent:
\begin{lemma}\label{lem:contact-exponent-and-blow-up}
  Assume $\Gamma$ is not invariant for $X$ and let
  $\pi:\mathcal{X}\rightarrow \cc$ be the blow-up with centre $(0,0)$ and
  $\overline{\Gamma}$ the strict transform of $\Gamma$ by $\pi$, which meets
  $\pi^{-1}(0,0)$ at $P$. Let $\overline{X}$ be pull-back of $X$ to
  $\mathcal{X}$. Let $n$ be the multiplicity of $\Gamma$ at $(0,0)$ and
  $\overline{n}$ that of $\overline{\Gamma}$ at $P$. Then:
  \begin{itemize}
  \item Either $\overline{X}$ is non-singular at $P$ and $(X,\Gamma)_{(0,0)}=n$
  \item Or $\overline{X}$ is singular at $P$ and
  \begin{equation*}
    (X, \Gamma)_{(0,0)} = (\overline{X}, \overline{\Gamma})_P + \overline{n}.
  \end{equation*}
  \end{itemize}
\end{lemma}
\begin{proof}
  If $\overline{X}$ is non-singular at $P$, the result is straightforward as $X$
  does not fix the tangent cone of $\Gamma$. Assume, then, that $P$ is singular
  for $\overline{X}$.

  Take a prepared irreducible Puiseux parametrization of $\Gamma_{\epsilon}$:
  \begin{equation}
    \label{eq:puiseux-gamma-epsilon}
    \Gamma_{\epsilon}\equiv \varphi_{\epsilon}(t) =
    \left(t^n, \sum_{m\leq i<j}a_it^i + \alpha_j(\epsilon)t^j + \hot\right)
  \end{equation}
  with $j=(X,\Gamma)_{(0,0)}$, as in Proposition
  \ref{pro:curve-and-puiseux-deformation-and-field}. Let
  $\pi:\mathcal{X}\rightarrow \cc$ be the blow-up with centre $(0,0)$ with
  equations $x=\overline{x}, y=\overline{x}\overline{y}$, for which
  $\overline{\Gamma}$ meets $\pi^{-1}(0,0)$ at $\overline{y}=0$. There are two
  cases:
  \begin{itemize}
  \item If $m\geq 2n$, the curve $\overline{\Gamma}_{\epsilon}$ has the same
    Puiseux parametrization as \eqref{eq:puiseux-gamma-epsilon} except that the
    $\overline{y}-$coordinate has all the exponents subtracted by $n$. The
    multiplicity of $\Gamma_{\epsilon}$ (and hence $\overline{\Gamma}$) is $n$
    and the result follows.
  \item If $n < m<2n$, then, $a_m\neq 0$ and, by Remark
    \ref{rem:contact-less-than-m-is-multiple-of-n}, we have $j\geq m$
    (otherwise, $j=n$ and $\overline{X}$ would not be singular at $P$) and we
    can write
    \begin{equation*}
      \overline{\Gamma}_{\epsilon} \equiv \overline{\varphi}_{\epsilon}(t) =
      (t^n, \sum_{m\leq i <j}a_it^{i-n} + a_j(\epsilon) t^{j-n} + \hot)
    \end{equation*}
    (with either $j>m$ and $a_m\neq 0$ or $j=m$ and $a_j(0)\neq 0$) which is not
    of irreducible Puiseux type (as $m-n<n$). In order to transform it to an
    irreducible Puiseux parametrization, one needs to extract $m-n-$th roots of
    the second coordinate:
    \begin{equation*}
      u = \sqrt[m-n]{\sum_{i=m}^{\infty} a_i (\epsilon) t^{i-n}}. 
    \end{equation*}
     Notice that such root is a holomorphic function defined in a neighborhood of
    $(\epsilon,t)=(0,0)$ since $a_m (0) \neq 0$. 
    We obtain
    \begin{equation*}
      t = \sum_{1\leq i < j-m+1} \alpha_iu^i + \alpha_j(\epsilon)u^{j-m+1}+\hot
    \end{equation*}
    where  $\alpha_j(0)\neq 0$ if $j=m$. From this, an irreducible
    Puiseux parametrization of $\overline{\Gamma}_{\epsilon}$ is given by
    \begin{equation*}
      \begin{split}
        &\overline{\Gamma}_{\epsilon} \equiv \overline{\varphi}_{\epsilon}(u) = \\
        &\left(
        \sum_{n\leq i < j-(m-n)} \overline{a}_i u^i +
        \overline{a}_j(\epsilon) u^{j-(m-n)} + \hot, u^{m-n}\right).
      \end{split}
    \end{equation*}
    with $\overline{a}_j(\epsilon)$ not constant. Hence,
    $(X,\Gamma)_{(0,0)} = (\overline{X}, \overline{\Gamma})_{P}+m-n$ and
    $\overline{n}$ is, in this case, $m-n$, which finishes the proof.
  \end{itemize}
\end{proof}

Lemma \ref{lem:contact-exponent-and-blow-up} states that
if $P_0=(0,0),P_1,\dots,P_N$ is the shared path between $X$ and $\Gamma$, then
\begin{equation*}
  (X, \Gamma)_{(0,0)} =
  \left\{
    \begin{array}{ll}
      \nu_0(\Gamma) & \mathrm{if}\, N=1 \\
      \nu_{P_1}(\Gamma_1 ) + (\overline{X}_1, \Gamma_1)_{P_1} & \mathrm{otherwise}
    \end{array}
  \right.
\end{equation*}
where $\overline{X}_1$ and $\Gamma_1$ are, respectively, the pull-back of $X$
and the strict transform of $\Gamma$ at $P_1$. Hence:
\begin{corollary}
\label{cor:contact-exponent}
  Let $P_0,\dots,P_N$ be the shared path between $X$ and $\Gamma$. Then 
  \begin{equation}
    (X,\Gamma)_{(0,0)} = n_{N-1} + \sum_{j=1}^{N-1}n_j
  \end{equation}
 where $n_0, n_1, \dots, n_N$ is the sequence of multiplicities of $\Gamma$ at the points 
 of the path it shares with $X$.
  So that the contact order between $X$ and $\Gamma$ depends only on $\Gamma$
  and $N$.
\end{corollary}
Furthermore, the contact order of a vector field $X$ with
a branch $\Gamma$ is essentially that of the dual differential form with
$\Gamma$:

\begin{theorem}\label{the:contact-exponent-is-contact-minus-n}
  Let $X=A(x,y)\pt{x} + B(x,y)\pt{y}$ be a singular vector field at $\cc$ and
  $\Gamma$ be an analytic branch at $\cc$ with multiplicity $n$, not invariant
  for $X$. Let $\omega = -B(x,y)dx+A(x,y)dy$ be the ``dual'' differential form
  of $X$. Then
  \begin{equation*}
    \upsilon_{\Gamma}(\omega) = (X,\Gamma)_{(0,0)} + n.
  \end{equation*}
\end{theorem}
The result is an immediate consequence of Corollaries 
\ref{cor:kahlercontact-exponent-and-blow-up} and \ref{cor:contact-exponent}.
 
Corollary~\ref{cor:elimination-of-suitable-divisors} below will essentially
provide the analytic classification of plane branches except for Zariski's
invariant, which requires a specific definition. Consider an analytic branch
$\Gamma$ having irreducible Puiseux expansion
  \begin{equation*}
    \Gamma\equiv \varphi(t) = (x(t),y(t)) =
    \bigg(t^n,\, \sum_{i\geq n}a_it^i\bigg)
  \end{equation*}
  and let $X$ be a singular vector field at $\cc$ such that
  $(X,\Gamma)_{(0,0)}=j$. Let $\pi:\mathcal{X}\rightarrow \cc$ be the sequence
  of blow-ups producing the path $(P_i)_{i=0}^N$ shared by $X$ and $\Gamma$,
  where each $P_i$ (for $i=1\dots, N$) belongs to the irreducible component
  $E_i$ of the exceptional divisor $\pi^{-1}(0,0)$. We require some lemmas.  The
  next result can be seen as a corollary of the Poincar\'e-Hopf formula.

\begin{lemma}\label{lem:at-most-two-singular-points} 
  Let $X$ be a singular analytic vector field at $\cc$ and
  $\pi:\mathcal{X}\rightarrow \cc$ a finite sequence of blow-ups whose
  centres are singular points for each pull-back of $X$. If $E$ is an
  irreducible component of the exceptional divisor in $\pi^{-1}(0,0)$ which is
  not composed of singular points of the pull-back $\overline{X}$ of $X$ by
  $\pi$, then there are exactly two singular points for
  $\overline{X}$ in $E$ counting multiplicities.
\end{lemma}
  \begin{lemma}\label{lem:behaviour-of-parametric-puiseux-under-bup}
    If the pull-back $\overline{X}$ of $X$ to $\mathcal{X}$ has a single
    singularity in $E_N$, then $\overline{X}_{\vert_{E_N}}$ is a constant vector
    field away from that singular point.
  \end{lemma}
  \begin{proof}
    The vector field $\overline{X}_{\vert_{E_N}}$ has a singular point with
    multiplicity $2$ by Lemma \ref{lem:at-most-two-singular-points}.  It is
    analytically conjugated to $\partial / \partial z$ where $z$ is a complex
    coordinate in the chart $\mathbb{P}^1\mathbb{C} \setminus \{\infty\}$ of
    $\mathbb{P}^1\mathbb{C}$.
\end{proof}

\begin{lemma}\label{lem:arrive-to-contact-exponent}
  With the setting above, assume $N>1$ (or $(X,\Gamma)_{(0,0)}>n$, which is the
  same thing).  Let $\tilde{\Gamma}$ be another singular branch at $\cc$,
  topologically equivalent to $\Gamma$, admitting a parametrization
  \begin{equation*}
    \tilde{\Gamma}\equiv \tilde{\varphi}(t) = (\tilde{x}(t),\tilde{y}(t)) =
    \bigg(t^n,\, \sum_{i\geq n}\tilde{a}_it^i\bigg).
  \end{equation*}
  Let $(P_i)_{i=0}^{N}$, $(\tilde{P}_i)_{i=0}^{\tilde{N}}$ be the paths shared
  by $X$ and $\Gamma$, $\tilde{\Gamma}$, respectively. Then:
  $\tilde{a}_i =a_i \xi^{i}$ for some $\xi \in {\mathbb C}$ with $\xi^{n} =1$
  and any $n\leq i <j=(X,\Gamma)_{(0,0)}$ if and only if $N=\tilde{N}$ and also
  $P_k=\tilde{P}_k$ for $k=0,\dots,N-1$.
\end{lemma}
\begin{proof}
  The result is easily proved using an inductive argument similar to that of
  Lemma \ref{lem:contact-exponent-and-blow-up}, as the coefficients in the
  Puiseux parametrization of a branch are determined by its infinitely near
  points.  Let us clarify the role of the property $N=\tilde{N}$.  The divisor
  $E_N$ has exactly a corner point that is necessarily the unique singular point
  of $\overline{X}$ in $E_N$. The condition on the coefficients of $\Gamma$ and
  $\tilde{\Gamma}$ implies that none of these curves meet $E_N$ in the corner
  point and hence $P_N$ and $\tilde{P}_N$ are regular points of $\overline{X}$
  and $N =\tilde{N}$.
\end{proof}
Next, we see that if $(X, \Gamma)_{(0,0)} = j$ then the term of order $j$ can be
eliminated from a Puiseux parametrization of $\Gamma$ as long as $X$ has a
single singularity on the last divisor of the shared path.
 
\begin{corollary}\label{cor:elimination-of-suitable-divisors}
  With the same setting, if $\overline{X}$ has a single singularity in $E_N$
  then $\Gamma$ is analytically equivalent via a unique 
  local diffeomorphism in the one parameter group generated by $X$
  to a branch $\tilde{\Gamma}$ such that
  \begin{equation*}
    \tilde{\Gamma}\equiv \tilde{\varphi}(t) = (\tilde{x}(t),\tilde{y}(t))=
    \bigg(t^n,\, \sum_{i\geq n}\tilde{a}_it^i\bigg)
  \end{equation*}
  with $\tilde{a}_i = a_i$ for $n\leq i<j$ and $\tilde{a}_j=0$.
\end{corollary}
\begin{proof} 
  Let $R$ be the corner point of $E_N$.  Let
  ${\{ \psi_s \}}_{s \in {\mathbb C}}$ be the one parameter group associated to
  $X$.  We denote $\Gamma_s = \psi_{s} (\Gamma)$.  The Puiseux parametrization
  of $\Gamma_s$ is of the form
\begin{equation*}
    {\Gamma}_s \equiv {\varphi}_s (t) = ({x}_s (t),{y}_s (t)) =
    \bigg(t^n,\, \sum_{n \leq i < j } {a}_i t^i + \sum_{i \geq j } {a}_i (s) t^i  \bigg).
  \end{equation*}
where $j= (X,\Gamma)_{(0,0)}$ and $a_{j}(s)$ is not a constant function. 
Given $s, s' \in {\mathbb C}$ such that $s \neq s'$, the strict transforms of 
$\Gamma_{s}$ and $\Gamma_{s'}$ pass through different points of $E_{N} \setminus \{R\}$
by Lemma \ref{lem:behaviour-of-parametric-puiseux-under-bup}.
As a consequence the function $a_j$ is injective and hence
is also surjective. 
Thus there exists a unique $s_0 \in {\mathbb C}$ such that $a_j (s_0)=0$.
The curve $\tilde{\Gamma}$ is the curve $\Gamma_{s_0}$.
\end{proof}

\subsection{The moduli problem and holomorphic flows}
As an example of the relevance of our tools, we give a solution to
Zariski's moduli problem \cite{Zariski4} using flows instead of just analytic
diffeomorphisms. We first need some elementary results on the type of
singularities arising after a sequence of blow-ups of a singular analytic vector
field. Specifically, nilpotent vector fields only become regular on what are
called \emph{free exceptional divisors}. Notice that any vector field of
multiplicity at least $2$ is nilpotent since it has vanishing linear part. Then
we shall compute Zariski's $\lambda$ invariant \cite{Zariski-1966} as some
minimum of contacts between the branch and vector fields.

\strut

In this subsection, we fix a singular vector field $X$ at $\cc$ and a chain of
blow-ups
\begin{equation*}
    \mathcal{X}=\mathcal{X}_{N}\xrightarrow{\pi_{N-1}}
    \mathcal{X}_{N-1}\xrightarrow{\pi_{N-2}}\cdots
    \xrightarrow{\pi_1}\mathcal{X}_{1}
    \xrightarrow{\pi_0}\cc
\end{equation*}
each $\pi_i$ having centre $P_{i}$  belonging to $E_i=\pi_{i-1}^{-1}(P_{i-1})$,
the exceptional divisor corresponding to the blow-up of $P_{i-1}$. We call
$\overline{X}_i$ the pull-back of $X$ to $\mathcal{X}_i$, which we assume is
singular at $P_i$ for $i=0,\dots, N-1$ (writing $\overline{X}_0=X$ and $P_0=P$)
and we assume $P_N$ is a non-singular point of $\overline{X}_N$ in
$E_N=\pi^{-1}_{N-1}(P_{N-1})$.  We know that all the exceptional divisors
$E_1, \dots, E_N$ are invariant for $X_N$.

\begin{definition}\label{def:free-divisor}
  A divisor $E_i$ is \emph{free} if either $i=1$ or $P_{i-1}\in E_k$ implies
  $k=i-1$. In other words, if $E_i$ meets only one other exceptional divisor in
  $\mathcal{X}_i$, or what amounts to the same, if $P_{i-1}$ is not the
  intersection of two exceptional divisors.
\end{definition}

We need several technical results:
\begin{lemma}\label{lem:nilpotent-and-two-invariant-implies-mult-2}
  If $Y$ is a nilpotent singular vector field at $\cc$ admitting two transverse
  non-singular invariant curves then its multiplicity is strictly greater than
  $1$.
\end{lemma}
\begin{proof}
  Since $Y$ has two transverse non-singular invariant curves, its linear part
  must be diagonalisable. Since this linear part is nilpotent by hypothesis, it
  must be zero, i.e. $Y$ has multiplicity at least $2$.
\end{proof}

Nilpotent vector fields become regular only at free divisors:
\begin{lemma}\label{lem:nilpotent-implies-regular-at-free}
  With the previous notation, assume $X$ is nilpotent.  Then $E_N$
  (the divisor containing $P_N$, point at which $\overline{X}_N$ is regular) is
  a free divisor. Even more, there is only one singularity of $\overline{X}_N$
  in $E_N$ and if $N>1$, it is the intersection of $E_N$ with the only other
  divisor it meets (actually, $E_{N-1}$).
\end{lemma}
\begin{proof}
  If $N\leq 2$, then $E_N$ is automatically free and the statement holds. Assume
  then that $N>2$. A simple computation shows that $X_i$ has nilpotent linear
  part at $P_i$ for all $i=1,\dots, N-1$. If $E_N$ were not free, then $P_{N-1}$
  would belong to $E_{N-1}$ and another $E_k$ for $k\neq N-1$. As all the
  exceptional divisors are invariant, $P_{N-1}$ would be a singular point for
  $\overline{X}_{N-1}$ with two transverse non-singular invariant curves. We
  know that $\overline{X}_{N-1}$ has nilpotent linear part, hence $X_{N-1}$
  would have multiplicity at least $2$, by Lemma
  \ref{lem:nilpotent-and-two-invariant-implies-mult-2}. This prevents $P_N$ from
  being regular for $\overline{X}_N$, as the multiplicity of a singular vector
  field decreases at most by one after a single blow-up.

  The existence of a single singularity in $E_{N}$ is a consequence of the
  existence of a single eigenvector for the linear part of $\overline{X}_{N-1}$
  at $P_{N-1}$. If $N>1$ then the intersection of $E_N$ with the only other
  divisor it meets (which is, of necessity, $E_{N-1}$) must be a singular point
  for $\overline{X}_N$, as there are two invariant varieties through that point.
\end{proof}

  
A straightforward application of Lemma
\ref{lem:nilpotent-implies-regular-at-free} gives:
\begin{proposition}\label{pro:nilpotent-implies-shared-ends-in-free}
  Let $\Gamma$ be a branch through $\cc$ and $X$ be a  nilpotent
   singular vector field at $\cc$ for which $\Gamma$ is not
  invariant. Assume $(P_i)_{i=0}^N$ is the path shared by $\Gamma$ and $X$. Then
  $P_N$ is a non-singular point of $E_N$, $\overline{X}_N$ has a single singular
  point in $E_N$ and if $N>1$ then this singular point is $E_N\cap E_{N-1}$.
\end{proposition}
\begin{remark} 
\label{rem:exponent-in-semigroup-away}
Let $\Gamma$ be an analytic branch with irreducible Puiseux parametrization
  \begin{equation*}
    \Gamma\equiv \varphi(t) = (x(t),y(t)) =
    \bigg(t^n,\, \sum_{i\geq n}a_it^i\bigg).
  \end{equation*}
  We may assume $a_n=0$ up to replacing $\Gamma$ with
  $\mathrm{exp}\left( -a_{n} x \frac{\partial}{\partial y} \right)(\Gamma)$.
  Furthermore, since
  $(\Gamma, x^{k})_{(0,0)} = \left( x^{k} \partial/\partial y, \Gamma
  \right)_{(0,0)} =kn$ for $k \geq 1$, we may also assume that the Puiseux
  parametrization of $\Gamma$ is prepared by conjugating it with diffeomorphisms
  embedded in the one-parameter groups of the nilpotent vector fields
  $x^{2} \frac{\partial}{\partial y}, x^{3} \frac{\partial}{\partial y}, \hdots$
  by Corollary \ref{cor:elimination-of-suitable-divisors} and Proposition
  \ref{pro:nilpotent-implies-shared-ends-in-free}.  So in order to transform
  $\Gamma$ to normal form by using diffeomorphisms embedded in the flows of
  nilpotent vector fields, we may assume that the Puiseux parametrization is
  prepared.
\end{remark}
And, applying Corollary \ref{cor:elimination-of-suitable-divisors}, we get the
first \emph{elimination criterion} (in the sense of Zariski \cite{Zariski4} and
Hefez-Hernandes \cite{Hefez-Hernandes-classification}):

\begin{corollary}\label{cor:contact-exponent-for-nilpotent-goes-away}
  Let $\Gamma$ be an analytic branch with prepared irreducible Puiseux
  para\-me\-trization
  \begin{equation*}
    \Gamma\equiv \varphi(t) = (x(t),y(t)) =
    \bigg(t^n,\, \sum_{i\geq m}a_it^i\bigg)
  \end{equation*}
  and let $X$ be a  nilpotent  singular analytic vector field at $\cc$.
  Assume the contact exponent $j=(X,\Gamma)_{(0,0)}$ 
  between $X$ and $\Gamma$ is greater than $m$. 
  Then $\Gamma$ is analytically equivalent via a diffeomorphism
  in the holomorphic flow
  associated to $X$ to a branch $\tilde{\Gamma}$ with parametrization
    \begin{equation*}
    \tilde{\Gamma}\equiv \tilde{\varphi}(t) = (\tilde{x}(t),\tilde{y}(t))=
    \bigg(t^n,\, \sum_{i\geq m}\tilde{a}_it^i\bigg)
  \end{equation*}
  with $\tilde{a}_i = a_i$ for $i<j$ and $\tilde{a}_j=0$.
\end{corollary}

Thus, we can eliminate any finite number of coefficients (at least those up to
the conductor) from the Puiseux expansion of a curve $\Gamma$ as long as their
exponents correspond to the contact with a nilpotent vector field.
In order to give the complete classification, we only need to study
what happens if $X$ has multiplicity $1$ and if we can eliminate an infinite
family of exponents (the tail of the Puiseux expansion) with a single vector
field. Previously, though, note that the terms  whose exponent belongs to
the semigroup of $\Gamma$ can be removed from a Puiseux expansion via a
holomorphic flow:

\begin{corollary}\label{cor:exponent-in-semigroup-away}
  Let $\Gamma$ and $\varphi(t)$ be as in Corollary
  \ref{cor:contact-exponent-for-nilpotent-goes-away}. If $j > m$ is the
  intersection multiplicity of a singular analytic curve $\Delta$ with $\Gamma$
  then there exists a singular vector field $X$, with vanishing linear part, such
  that $(X,\Gamma)_{(0,0)}=j$. In particular the same conclusion as in Corollary
  \ref{cor:elimination-of-suitable-divisors} holds (i.e. the term $a_j$ can be
  eliminated from the parametrization $\varphi$ without affecting the previous
  ones).
\end{corollary}
\begin{proof}
  Consider $f \in {\mathbb C}\{x,y\}$ such that
  $j= (\Gamma,f)_{(0,0)}$. We obtain that the multiplicity at the origin is greater than $1$
  since otherwise $(\Gamma, f)_{(0,0)} \leq m$.   
  The vector field $X=f(x,y)\pt{y}$ has vanishing linear part at $(0,0)$. 
  By Theorem
  \ref{the:contact-exponent-is-contact-minus-n}, this $X$ has contact exponent
  $(X,\Gamma)_{(0,0)}=j$.  Applying Corollary
  \ref{cor:contact-exponent-for-nilpotent-goes-away}, we are done.
\end{proof}

From all the previous discussions, we may assume that, after a finite
composition of local diffeomorphisms embedded in flows (including a linear one,
intended to make the tangent cone of $\Gamma$ at $(0,0)$ different from the $OX$
axis), $\Gamma$ has a prepared Puiseux expansion of the form
\begin{equation}
  \label{eq:puiseux-without-semigroup-1}
  \Gamma\equiv \varphi(t) = (x(t),y(t)) =
  \bigg(t^n,\,   \sum_{i  \geq  m}a_it^i\bigg)
\end{equation}
where $n<m$, $n \nmid m$, $a_m \neq 0$ and if $i$ is in the semigroup associated
to $\Gamma$ and $m < i\leq c$ where $c$ is the conductor of $\Gamma$, then
$a_i=0$. We may also assume that $a_i=0$ if $i\leq c$ is the contact exponent
with a nilpotent vector field.  We shall deal with the co-final terms later on
(it is well known that they may be eliminated with a single analytic
diffeomorphism anyway).
\begin{proposition}
\label{pro:lambda}
In the conditions of the last paragraph, let $\lambda$ be the least exponent
$\lambda>m$ such that $a_{\lambda}\neq 0$ and $c$ the conductor of
$\Gamma$. Assume $\lambda < c$. Let $X$ be a non-nilpotent singular vector
field.  Then $(X, \Gamma)_{(0,0)} \leq \lambda$.  Moreover
$(X, \Gamma)_{(0,0)} < \lambda$ implies that $(X, \Gamma)_{(0,0)}$ is of the
form $(pn +qm) - n$ where $p \geq 0$, $q \geq 0$ and
$m-n \neq (pn +qm) - n \geq n$.
\end{proposition}
\begin{proof}
  Let $X = A(x,y) \frac{\partial}{\partial x} + B(x,y) \frac{\partial}{\partial y}$.
  Write $A(x,y)=a_{10}x + a_{01}y + \overline{A}(x,y)$ and
  $B(x,y)=b_{10}x + b_{01}y + \overline{B}(x,y)$.
  Consider $\varphi^{\ast}\omega$, where $\omega$ is the
  dual form $\omega=-B(x,y)dx+A(x,y)dy$:
  \begin{equation*}
    \begin{split}
      \varphi^{\ast}&\omega =-(b_{10}t^n + b_{01}(a_m t^m +
      a_{\lambda}t^{\lambda} + \hot)+\overline{B}(\varphi(t)))nt^{n-1}dt
      +\\
      &(a_{10}t^n + a_{01}(a_m t^m + a_{\lambda}t^{\lambda} +
      \hot)+\overline{A}(\varphi(t)))(m a_m t^{m-1} + \lambda
      a_{\lambda}t^{\lambda-1} + \hot)dt.
    \end{split}
  \end{equation*}
  Let $j= (X, \Gamma)_{(0,0)}$, so that $j+n=\nu_{\Gamma}(\omega)$ by Theorem
  \ref{the:contact-exponent-is-contact-minus-n}.  We may assume
  $j \not \in \{n,m\}$ since $n = (2n + 0 \cdot m)- n$ and $m = (n + m) -n$.

  We have $b_{10} = 0$ since otherwise $\nu_{\Gamma}(\omega)= 2n$ and we would
  have $j=n$.  The property $\nu_{\Gamma}(\omega)< n+m$ implies that
  $\nu_{\Gamma}(\omega)$ is a multiple of $n$. In particular we get
  $\nu_{\Gamma}(\omega) \neq m$ and $j \neq m-n$. Moreover $j$ is a multiple of
  $n$ greater or equal than $n$. So we may assume
  $\nu_{\Gamma}(\omega) \geq n+m$ from now on.  Indeed we obtain
  $\nu_{\Gamma}(\omega) > n+m$ and $j >m$ since $j \neq m$.  This implies
  $b_{01} n = a_{10} m$. Since $X$ is non-nilpotent, it follows that
  $a_{10} \neq 0$ and $b_{01} \neq 0$.  The pull-back $\varphi^{\ast}\omega$
  satisfies
 \begin{equation*}
 t  \varphi^{\ast}\omega = (g (t) + (a_{10} \lambda - b_{01} n) a_{\lambda} t^{n+\lambda} +
 O(t^{n+\lambda +1})) dt
 \end{equation*}
 where $a_{10} \lambda - b_{01} n \neq 0$ and the exponents of all monomials
 with non-vanishing coefficients of the Taylor power series expansion of $g(t)$
 belong to the semigroup
 \begin{equation*}
 S':= \{ pn+qm: p \geq 0, \ q \geq 0, \ p+q \geq 1 \}.
\end{equation*}
We claim that $n+\lambda$ does not belong to $S'$. Otherwise
$n+\lambda = pn +qm$. If $p \geq 1$ then $\lambda$ belongs to $S$, a
contradiction. If $p=0$ then $q\geq 2$ and $\lambda$ is the contact order
$(y^{q-1} \partial/\partial x, \Gamma)_{(0,0)}$ of $\Gamma$ with a nilpotent
vector field, again a contradiction.

Since $\lambda +n \not \in S'$ and $a_{10} \lambda - b_{01} n \neq 0$, it
follows that $m < j \leq \lambda$.  Moreover $j < \lambda$ implies
$j \in S' - n$.
\end{proof}
\begin{theorem}\label{the:elimination-of-terms-greater-than-lambda}
  In the same conditions as above, if $j>\lambda$ is the contact exponent of
  $\Gamma$ with an analytic vector field $X$, then the term of order $j$ can be
  eliminated from a prepared Puiseux expansion via a diffeomorphism in a 
  nilpotent holomorphic flow.
\end{theorem}
\begin{proof} 
Let $X$ be a singular vector field such that $(X,\Gamma)_{(0,0)} = j$.
Since $j > \lambda$, $X$ is nilpotent  by
Proposition \ref{pro:lambda}.  Apply Corollary
  \ref{cor:contact-exponent-for-nilpotent-goes-away} to finish the proof. 
\end{proof}
 We end this section with a characterization of Zariski's $\lambda$
invariant of a plane branch $\Gamma$ in terms of tangency orders (or contact
orders) of vector fields with $\Gamma$.  
\begin{theorem}[Zariski's $\lambda$ invariant]
\label{teo:Zar-inv}
In the conditions of Proposition \ref{pro:lambda}, let $\lambda$ be the least
exponent $\lambda>m$ such that $a_{\lambda}\neq 0$ and $c$ the conductor of
$\Gamma$. Then $\lambda+n=\upsilon_{\Gamma}(mydx-nxdy)$ holds.  Indeed
if $\lambda < c$ then $m$ and $\lambda$ are the unique positive integers $j$
such that $j$ is the contact exponent of a singular vector field with $\Gamma$
but is not the contact exponent of a nilpotent vector field with $\Gamma$.
 As a consequence, $\lambda$ is an analytic invariant of $\Gamma$ if
$\lambda<c$.
\end{theorem}
\begin{proof}
  Let $\lambda<\infty$ be as in the statement and $\omega=mydx-nxdy$. By direct
  substitution:
  \begin{equation*}
    \upsilon_{\Gamma}(\omega) = \ord_{t}\bigg(
      \big(  \sum_{i \geq m}mna_it^{i+n-1}  \big) -
      \big( \sum_{i \geq m}nia_it^{i+n-1}  \big)
    \bigg)+1
  \end{equation*}
  so that
  \begin{equation*}
    \upsilon_{\Gamma}(\omega) =  
    {\left(  nx \frac{\partial}{\partial x} +  my \frac{\partial}{\partial y}, \Gamma \right)}_{(0,0)}=
  \ord_t((mn-n\lambda)a_m t^{\lambda+n-1}+\hot) +1 =\lambda+n,
  \end{equation*}
  which gives the first part of the statement.  Moreover $m$ is also a contact exponent since
  $m= (x \partial /\partial x, \Gamma)_{(0,0)}$.
  It can not be expressed as a contact with a nilpotent vector field, since the coefficient
  of $t^m$ in $y(t)$ can not be erased (Corollary \ref{cor:contact-exponent-for-nilpotent-goes-away}).

  Assume that there exists $j \not \in \{m,\lambda\}$ satisfying the hypotheses.
  We obtain $j < \lambda$ and $j+n = (pn+qm) \geq 2n$ for some $p \geq 0$ and $q \geq 0$
with $(p,q) \not \in \{ (1,0), (0,1), (1,1)\}$  by Proposition \ref{pro:lambda}. The vector field
$x^{p-1} y^{q} \partial / \partial y$ is nilpotent if $p \geq 1$ since $(p,q) \neq (1,0)$
and $(p,q) \neq (1,1)$.
It satisfies $(x^{p-1} y^{q} \partial / \partial y, \Gamma)_{(0,0)} =j$ and so we get a contradiction.
Analogously if $q \geq 1$ the vector field $x^{p} y^{q-1} \partial / \partial x$
is nilpotent and $(x^{p} y^{q-1} \partial / \partial x, \Gamma)_{(0,0)} =j$ holds, providing a
contradiction. 
\end{proof}

\section{Analytic classes and their completeness}
 Now we focus on whether curves in the same class of analytic conjugacy of
 a given plane branch are conjugated by local diffeomorphisms in
 a one-parameter group.
 Let us give some definitions.
 \begin{definition}
 We say that $\rho (t)$ (where $\rho(t)$ belongs to the maximal ${\mathfrak m}_1$
 ideal of ${\mathbb C}[[t]]$) is a formal diffeomorphism
 if its linear part is non-vanishing (or in other words if
 $\rho (t) \in {\mathfrak m}_1 \setminus {\mathfrak m}_1^{2}$). If, in addition
 to the previous properties, $\rho(t)$ belongs to ${\mathbb C}\{t\}$ then is a
 local biholomorphism defined in the neighborhood of the origin by the inverse
 function theorem.

 We say that
 $\psi(x,y) = (a(x,y), b(x,y)) \in {\mathfrak m} \times {\mathfrak m}$ is a
 formal diffeomorphism and we denote
 $\psi \in \widehat{\mathrm{Diff}} ({\mathbb C}^{2},0)$ if the linear part of
 $\psi$ at the origin is a linear isomorphism.
 \end{definition}
 \begin{definition}
   The Krull topology for formal diffeomorphisms or vector fields is defined by
   considering them as $n$-uples of formal power series and the induced product
   topology in ${\mathbb C}[[x,y]]^{n}$ (cf. Definition \ref{def:krull}).
 \end{definition}
 \begin{definition}
   Let $\psi$ be a formal diffeomorphism. We say that $\psi$ is {\it unipotent}
   if its linear part is a unipotent linear map.
 \end{definition}
 \begin{definition}
   The time $1$ flow $\mathrm{exp} (X)$ of a singular vector field $X$ is
 \begin{equation*}
 \mathrm{exp} (X)(x,y) = \left( \sum_{j=0}^{\infty} \frac{X^{j} (x)}{j!},
 \sum_{j=0}^{\infty} \frac{X^{j} (y)}{j!} \right).
 \end{equation*}
 Given a formal vector field $X$ we use the previous formula to define the
 formal diffeomorphism $\mathrm{exp}(X)$. The formula is well-defined: indeed,
 if $(X_k)_{k \geq 1}$ is a sequence of vector fields that converges to $X$ in
 the ${\mathfrak m}$-adic topology then $\mathrm{exp}(X_k)$ converges to
 $\mathrm{exp}(X)$ in the Krull topology when $k \to \infty$.
 \end{definition}
 First let us provide an example of a complete class.
 \begin{proposition}
 \label{pro:smooth-is-complete}
 The class of analytic conjugacy of all smooth plane branches is complete.
 \end{proposition}
 \begin{proof}
   Consider two smooth curves $\Gamma$, $\Gamma'$. Up to a change of coordinates
   we may assume $\Gamma \equiv (y=0)$ and that $\Gamma'$ is not tangent to
   $x=0$. Thus, $\Gamma$ and $\Gamma'$ admit Puiseux parametrizations
   $(t,0)$ and $(t,a(t))$, respectively. As a consequence, the local
   diffeomorphism $\psi (x,y) = (x,y+ a(x))$ conjugates $\Gamma$ and $\Gamma'$.
   Since $\psi = \mathrm{exp} ( a(x) \partial / \partial y)$ we are done.
 \end{proof}
The following theorem implies Theorem \ref{teo:forcomplete}.
\begin{theorem}
\label{teo:formal-complete}
Let $\Gamma$ and $\Gamma'$ be two plane branches that are conjugated by the
exponential $\mathrm{exp}(\hat{X})$ of a singular formal vector field. Then they
are conjugated by a local diffeomorphism embedded in the flow of a singular
holomorphic vector field $X$. Moreover if $\hat{X}$ is nilpotent we may assume
that $X$ is nilpotent.
\end{theorem}
 \begin{proof}
   Assume $\Gamma \neq \Gamma'$ since the result is trivial otherwise.  Up to a
   linear change of coordinates we may assume that none of the tangent cones of
   the curves $\Gamma$ and $\Gamma'$ is the $OY$ axis.  The curves $\Gamma$ and
   $\Gamma'$ have Puiseux parametrizations $\alpha (t) = (t^{n}, a(t))$ and
   $\beta (t) = (t^n, b(t))$ respectively where $n$ is the common multiplicity
   at the origin. By hypothesis, there exists
   $\rho \in \widehat{\mathrm{Diff}} ({\mathbb C},0)$ such that
   $(\mathrm{exp}(\hat{X}) \circ \alpha)(t) \equiv (\beta \circ \rho)(t)$.

   Write
   $\hat{X}= \hat{A} (x,y) \frac{\partial}{\partial x} + \hat{B} (x,y)
   \frac{\partial}{\partial y}$.  Consider a sequence $(X_k)_{k \geq 1}$ of
   singular vector fields that converge to $\hat{X}$ in the ${\mathfrak m}$-adic
   topology. For instance this can be obtained by defining
   $X_k = A_k (x,y) \frac{\partial}{\partial x} + B_k (x,y)
   \frac{\partial}{\partial y}$ where $A_k$ (resp. $B_k$) is the polynomial of
   degree less or equal than $k$ such that
   $\hat{A} - A_k \in {\mathfrak m}^{k+1}$ (resp.
   $\hat{B} - B_k \in {\mathfrak m}^{k+1}$) for $k \geq 1$.  Analogously we
   choose a sequence $(\rho_k)_{k \geq 1}$ in $\mathrm{Diff} ({\mathbb C},0)$
   converging to $\rho$ in the Krull topology.  We define the curve $\Gamma_k$
   as $\mathrm{exp}(X_k) (\Gamma)$ and denote
   $(x_k (t), y_k (t))= (\mathrm{exp}(X_k) \circ \alpha \circ \rho_{k}^{-1})(t)$
   for $k \geq 1$.  The sequence $(x_k (t), y_k (t))_{k \geq 1}$ converges to
   $(t^{n}, b(t))$ in the Krull topology.  Consider $\sigma_k (t)$ the
   holomorphic function such that $\sigma_k (t)^{n} \equiv x_k (t)$ and
   $(\sigma_{k})' (0)=1$ for $k \gg1$.  Since $(\sigma_{k})' (0) \neq 0$, it is
   a local diffeomorphism and its inverse $\sigma_k^{-1}$ exists.  The
   sequence $(\sigma_k)_{k \geq 1}$ converges to $t$ in the Krull topology.
   Thus $(t^n, b_k (t)):= (x_k, y_k) \circ \sigma_k^{-1} (t)$ is a
   parametrization of $\Gamma_k$ that converges to $(t^{n}, b(t))$ in the Krull
   topology when $k \to \infty$.

   Since $\Gamma$ and $\Gamma_k$ are conjugated by a local diffeomorphism
   contained in a one-parameter flow, it suffices to show that for fixed
   $k \gg1$ there exists a local diffeomorphism
   $\theta \in \mathrm{Diff} ({\mathbb C}^{2},0)$ such that
   $\theta (\Gamma)= \Gamma$ and $\theta (\Gamma') = \Gamma_k$.  Indeed then
   $\mathrm{exp} (\theta^{*} X_k)(\Gamma)= \Gamma'$.  Moreover, if $\hat{X}$ is
   nilpotent then $X_k$ is nilpotent for any $k \geq 1$, since $\hat{X}$ and
   $X_k$ have the same linear part at $(0,0)$ and $\theta^{*} X_k$ is nilpotent
   as a conjugate of $X_k$.

   Thus, we need to prove that there exists $\theta$ with
  \begin{equation*}
    \left\{
      \begin{array}{l}
        (\theta\circ \alpha) (t) \equiv \alpha(t)\\
        \theta (t^n, b(t)) \equiv (t^n, b_k(t)).
      \end{array}
    \right.
  \end{equation*}
  
  Let $h(x,y)=0$ be a local (irreducible) equation of $\Gamma$ and define
  \begin{equation*}
    \theta(x,y) = (x, y + h(x,y)\gamma(x,y)).
  \end{equation*}
  If we prove that there exists a holomorphic function $\gamma(x,y)$ for which
  the conditions on $\theta$ are satisfied, we are done. The fact that
  $\theta(\alpha(t))\equiv \alpha(t)$ is obvious by construction. The other
  condition, $ \theta (t^n, b(t)) \equiv (t^n, b_k(t))$ is equivalent to
  \begin{equation}
  \label{equ:conjfor}
    \gamma (t^n, b(t)) \equiv \frac{b_k (t) - b(t)}{h(t^n, b(t))}.
  \end{equation}
  The denominator $h(t^n, b(t))$ is not identically $0$ since
  $\Gamma \neq \Gamma'$.  The right hand side of Equation (\ref{equ:conjfor})
  converges to $0$ in the Krull topology when $k \to \infty$. Thus there exists
  a solution of Equation (\ref{equ:conjfor}) for some $k >1$, which finishes the
  proof.
 \end{proof}
 \begin{corollary}
 \label{cor:embedding-unipotent}
 Let $\Gamma, \Gamma'$ be two plane branches conjugated by a unipotent formal diffeomorphism
$\psi \in \widehat{\mathrm{Diff}} ({\mathbb C}^{2},0)$. Then $\Gamma$ and $\Gamma'$ are
conjugated by a local diffeomorphism embedded in a one-parameter group generated by
a nilpotent vector field.
 \end{corollary}
We will use the next well-known result.
\begin{remark}[{cf. \cite{Ecalle, Martinet-Ramis2}}]
\label{rem:nil-to-unip}
The exponential provides a bijection between the set of formal nilpotent singular vector fields
and the set of unipotent formal diffeomorphisms. Moreover, it particularizes to a bijection between
the Lie algebra of formal vector fields with vanishing linear part at the origin 
and the group of formal diffeomorphisms with identity linear part.
\end{remark}
Corollary \ref{cor:embedding-unipotent} is an immediate consequence of
Theorem \ref{teo:formal-complete} and Remark \ref{rem:nil-to-unip}.
\begin{corollary}
\label{cor:2-geodesic}
Let $\Gamma$ and $\Gamma'$ be two plane branches in the same class of analytic conjugacy.
There exists a nilpotent vector field $X$ and a linear vector field $Y$ such that
$(\mathrm{exp}(Y) \circ \mathrm{exp}(X)) (\Gamma)=\Gamma'$.
\end{corollary}
\begin{proof}
Let $\psi \in {\mathrm{Diff}} ({\mathbb C}^{2},0)$ such that $\psi (\Gamma) = \Gamma'$.
Then we have $\psi = L \circ \sigma$ where $L$ is the linear part of $\psi$ at the origin
and the linear part of $\sigma$ at the origin is the identity.
Denote $\overline{\Gamma} = \sigma (\Gamma)$. We have
$L(\overline{\Gamma}) = \Gamma'$.
There exists a nilpotent vector field $X$ such that
$\mathrm{exp}(X)(\Gamma) = \overline{\Gamma}$ by Corollary \ref{cor:embedding-unipotent}.
Moreover, $L$ is of the form $\mathrm{exp} (Y)$ for some linear vector field.
Therefore we obtain $(\mathrm{exp}(Y) \circ \mathrm{exp}(X)) (\Gamma)=\Gamma'$.
\end{proof}

We have obtained the analytic reduction of holomorphic branches
\cite{Hefez-Hernandes-classification} to short parametrizations.
\begin{definition}
  Let $\Gamma$ be a germ of plane branch. We denote by $\Lambda$ the set of
  contact exponents between $\Gamma$ and singular vector fields.  Notice that
  $\Lambda + n$ is the set of orders of contact of K\"{a}hler differentials with
  $\Gamma$ by Theorem \ref{the:contact-exponent-is-contact-minus-n}.
\end{definition}

\begin{corollary}\label{cor:analytic-classification}
  Let $\Gamma$ be a branch at $\cc$ with  prepared 
   irreducible Puiseux parametrization
  \begin{equation*}
    \Gamma\equiv \varphi(t)=(t^n, \sum_{i\geq m}^{\infty}a_it^i).
  \end{equation*}
  Let $\lambda$ be its Zariski invariant (or $\lambda=\infty$) and $c>\lambda$
  the conductor of the semigroup associated to $\Gamma$. There is a
  nilpotent singular vector field $X$ 
  such that
  \begin{equation*}
    \mathrm{exp}(X)(\Gamma)\equiv \overline{\varphi}(t) =
    (t^n,  a_m 
    t^m + a_{\lambda}t^{\lambda} +  \sum_{i>\lambda}^{c-1}\overline{a}_it^i)
  \end{equation*}
  with $\overline{a}_i=0$ for $i \in \Lambda$. 
\end{corollary}

A parametrization like in Corollary \ref{cor:analytic-classification} is called,
a \emph{short parametrization}.
 \begin{proof}
 In order to simplify $\Gamma$
 we remove step by step coefficients of $t^j$  in the second component of the 
 Puiseux parametrization of $\Gamma$ for some $j >n$.
 The normalizing map is of the form $\mathrm{exp} (X_j)$ where $(X_j, \Gamma)_{(0,0)} =j$
 and $X_j$ is a nilpotent singular vector field by Corollary   \ref{cor:exponent-in-semigroup-away}
 and Theorem \ref{teo:Zar-inv}.
 Indeed it is easy to see that we may assume
 that $X_j \to 0$ when $j \to \infty$ in the ${\mathfrak m}$-adic topology.
 Moreover the tangent cone of $\Gamma$ defines a singular point $P$ of
 $\overline{(X_j)}_1$ in the divisor $E_1$ of the blow-up of the origin for $j >n$, since otherwise
 $n < j  = (X_j, \Gamma)_{(0,0)} = n$.
 We deduce that $\mathrm{exp}(X_j)$ is a unipotent diffeomorphism whose linear part has matrix
  \begin{equation*}
    \left(
    \begin{array}{cc}
    1 & c_j \\
    0 & 1
    \end{array}
    \right)
  \end{equation*}
  where $c_j=0$ if $j \gg1$.
 The limit of the composition of these exponentials,  in the appropriate order,
 is a well-defined formal unipotent diffeomorphism
 $\psi \in \widehat{\mathrm{Diff}} ({\mathbb C}^{2},0)$ conjugating $\Gamma$ with a curve
 with Puiseux parametrization of the form $\overline{\varphi}(t)$.
 The result is a consequence of Corollary \ref{cor:embedding-unipotent}.
\end{proof}
 \begin{remark}
   The reduction to normal form in \cite{Hefez-Hernandes-classification} is
   obtained via the action of unipotent diffeomorphisms.  We have just restated
   this fact in the context of holomorphic flows.
 \end{remark}
 The expression of Corollary \ref{cor:analytic-classification} can be simplified
 further by means of another flow (corresponding to a linear change of
 coordinates and a change of parameter):

\begin{lemma}[\cite{Zariski4}]\label{lem:canonical}
  A branch $\Gamma$ whose short parametrization is
  \begin{equation*}
    \Gamma\equiv \varphi(t) = (t^n,  a_m  t^m + a_{\lambda}t^{\lambda} +
    \sum_{i>\lambda}^{c-1}a_it^i)
      \end{equation*}
      is analytically equivalent to
      \begin{equation*}
        \Gamma^{\prime} \equiv (t^{n},
        t^{m}+t^{\lambda} + \sum_{i>\lambda}^{c-1}\overline{a}_it^i)
      \end{equation*}
      where there exist $u,v \in\mathbb{C}^{\star}$ such that
      $\overline{a}_i = v^{m} u^{-i} a_i$.
\end{lemma}
\begin{proof}
  We define $\psi (x,y) = (u^n x, v^m y)$ for some $u,v \in\mathbb{C}^{\star}$
  to be specified later on. We have
   \begin{equation*}
    (\psi \circ \varphi)(t) = (u^n t^n, v^m a_m t^m + v^m a_{\lambda}t^{\lambda} +
    \sum_{i>\lambda}^{c-1} v^m a_it^i).
   \end{equation*}
 Define the parameter $s=ut$.
The curve $\psi (\Gamma)$ has parametrization
 \begin{equation*}
      (s^n, v^m u^{-m} a_m s^m + v^m u^{-\lambda} a_{\lambda} s^{\lambda} +
    \sum_{i>\lambda}^{c-1} v^m u^{-i} a_i s^i).
   \end{equation*}
It suffices to consider $u,v \in\mathbb{C}^{\star}$ such that
$v^m u^{-m} = a_m^{-1}$ and $v^m u^{-\lambda}  = a_{\lambda}^{-1}$.
\end{proof}
Combining Corollary \ref{cor:analytic-classification}
and  Lemma \ref{lem:canonical} we obtain the following result.
\begin{corollary}\label{cor:equivalence-under-composition-of-flows}
  Let $\Gamma$ be a singular branch in $\cc$  
  having   conductor $c$. 
  Let $(x,y)$ be a local system of coordinates. There exist a local
  diffeomorphism $\psi$ embedded in the flow of a nilpotent vector field, a
  linear map $G$ and a reparametrization $\tau\in \mathrm{Diff}(\mathbb{C},0)$
  such that
  \begin{equation*}
    \big(G\circ\psi(\Gamma)\equiv
    G\circ\psi\circ\varphi \circ \tau \big)(t)=
    \left(t^n, t^m + t^{\lambda} + \sum_{\lambda<i<c}a_it^i\right)
  \end{equation*}
  where $\varphi(t)$ is the parametrization of $\Gamma$ with $a_i=0$ if $i<c$
  and $i \in \Lambda\setminus \left\{ \lambda \right\}$.
\end{corollary}
\begin{proof}
  There exists a linear automorphism $H(x,y)$ such that the tangent cone to
  $\Gamma': = H(\Gamma)$ at the origin is the axis $y=0$.  There exists a local
  diffeomorphism $J(x,y)= (x,y + c(x))$ for some $c(x) \in {\mathbb C}\{x\}$ of
  vanishing order at least $2$ such that $J(\Gamma')$ has a prepared irreducible
  Puiseux parametrization. We apply Corollary \ref{cor:analytic-classification}
  to $J(\Gamma')$ to obtain a unipotent diffeomorphism
  $\phi \in {\mathrm{Diff}} ({\mathbb C}^{2},0)$ such that $\phi (J(\Gamma'))$
  has a short parametrization.  Finally we apply Lemma \ref{lem:canonical} to
  $\phi (J(\Gamma'))$ to obtain a linear isomorphism $K$ such that
  $\Gamma'':=K(\phi (J(\Gamma')))$ has the desired parametrization. The
  diffeomorphism $\phi \circ J$ is unipotent since the linear part $D_0 J$ of
  $J$ at the origin is the identity map.  The conjugate
  $H^{-1} \circ (\phi \circ J) \circ H$ of $\phi \circ J$ is a unipotent
  diffeomorphism $\rho \in {\mathrm{Diff}} ({\mathbb C}^{2},0)$ and then we
  obtain $\Gamma'' = (G \circ \rho)(\Gamma)$ where $G= K \circ H$ is a linear
  map.  Since $\Gamma$ and $\rho (\Gamma)$ are conjugated by a unipotent local
  diffeomorphism, it follows that they are conjugated by a local diffeomorphism
  $\psi$ embedded in the flow of a nilpotent vector field. We obtain
  $\Gamma'' = (G \circ \psi)(\Gamma)$.
\end{proof}
The parametrization provided by Corollary
\ref{cor:equivalence-under-composition-of-flows} is called a \emph{canonical
  parametrization} by Zariski \cite{Zariski4} and the \emph{normal form of
  $\Gamma$} by Hefez-Hernandes \cite{Hefez-Hernandes-classification}. We shall
use the latter terminology.  Moreover, if $\overline{\Gamma}$ is another branch
whose normal form has coefficients $\overline{a}_i$, one can prove (see
\cite{Zariski4} and \cite{Hefez-Hernandes-classification}) that they are
analytically equivalent if and only if there exists $u$ such that
$u^{\lambda-m}=1$ and $\overline{a}_i=u^{i-m}a_i$, which describes the complete
moduli of $\Gamma$.
\section{Non-complete analytic classes}
We provide examples of non-complete analytic classes. Whether or not a single
formal diffeomorphism is embedded in the flow of a formal singular vector field
is deeply related to the spectrum of its linear part and more precisely to the
resonances among its eigenvalues.  For the sake of completeness we recall these
concepts along with some results.  We work in dimension $2$ because that is the
case we are interested in but the results concerning resonances are valid for
any dimension (cf. \cite{ilyashenko-yakovenko-lectures}).

\begin{definition}
Consider a formal singular vector field $X$ whose linear part is in Jordan normal form, in particular
$X$ is of the form
\begin{equation*}
X = \left(  \lambda_1 x + \delta y + \sum_{i+j \geq 2} a_{ij} x^{i} y^{j} \right) \frac{\partial}{\partial x} +
\left(\lambda_2 y + \sum_{i+j \geq 2} b_{ij} x^{i} y^{j} \right) \frac{\partial}{\partial y} .
\end{equation*}
We say that the monomial $x^{i} y^{j} \partial / \partial x$ with
$i \geq 0$, $j \geq 0$, $i+j \geq 1$ and $(i,j) \neq (1,0)$ is resonant if
$i \lambda_1 + j \lambda_2 = \lambda_1$. Analogously
we say that the monomial $x^{i} y^{j} \partial / \partial y$ with
$i \geq 0$, $j \geq 0$, $i+j \geq 1$ and $(i,j) \neq (0,1)$ is resonant if
$i \lambda_1 + j \lambda_2 = \lambda_2$.
\end{definition}
\begin{definition}
Consider a formal diffeomorphism $\psi$ whose linear part is in Jordan normal form, in particular
$\psi$ is of the form
\begin{equation*}
\psi(x,y) = \left(  \lambda_1 x + \delta y + \sum_{i+j \geq 2} a_{ij} x^{i} y^{j} ,
\lambda_2 y + \sum_{i+j \geq 2} b_{ij} x^{i} y^{j} \right) .
\end{equation*}
We say that the monomial $x^{i} y^{j} e_1 := (x^i y^j, 0)$ with
$i \geq 0$, $j \geq 0$, $i+j \geq 1$ and $(i,j) \neq (1,0)$ is resonant if
$\lambda_1^{i}  \lambda_2^{j} = \lambda_1$. Analogously
we say that the monomial $x^{i} y^{j} e_2 := (0,x^i y^j)$ with
$i \geq 0$, $j \geq 0$, $i+j \geq 1$ and $(i,j) \neq (0,1)$ is resonant if
$\lambda_1^{i}  \lambda_2^{j} = \lambda_2$.
A formal diffeomorphism is non-resonant if there are no resonant monomials.
\end{definition}
\begin{remark}
The property of being non-resonant depends only on the eigenvalues of the linear part.
\end{remark}
The next result is Poincar\'{e}'s linearisation map for formal
diffeomorphisms. As is customary, we denote $D_0\psi$ the linear part of a local
diffeomorphism $\psi$.
\begin{proposition}[{cf. \cite[Theorem 4.21]{ilyashenko-yakovenko-lectures}}]
\label{pro:Poincare-linear}
Let $\psi \in \widehat{\mathrm{Diff}} ({\mathbb C}^{2},0)$ be a non-resonant
formal diffeomorphism.  Then $\psi$ is conjugated by a formal diffeomorphism to
$(x,y) \mapsto (\lambda_1 x, \lambda_2 y)$ where $\lambda_1$ and $\lambda_2$ are
the eigenvalues of the linear part $D_0 \psi$ of $\psi$ at $(0,0)$.
\end{proposition}
\begin{corollary}
\label{cor:Poincare-linear}
Let $\psi \in \widehat{\mathrm{Diff}} ({\mathbb C}^{2},0)$ be a non-resonant
formal diffeomorphism.  Then there exists a formal singular vector field $X$
such that $\psi = \mathrm{exp}(\hat{X})$.
\end{corollary}
\begin{proof}
The formal diffeomorphism $\psi$ is formally conjugated to a linear diagonal map by
Proposition \ref{pro:Poincare-linear}. Since the latter map
is embedded in the flow of a singular vector field, it follows that
$\psi$ is embedded in the flow of a formal vector field.
\end{proof}
Let us consider the problem of embedding resonant diffeomorphisms in formal
flows.  Let $\psi \in \widehat{\mathrm{Diff}} ({\mathbb C}^{2},0)$ and assume
for simplicity that $(D_0 \psi)(x,y) =(\lambda_1 x, \lambda_2 y)$.  The equation
$\psi = \mathrm{exp} (X)$ implies $D_0 \psi = \mathrm{exp} (D_0 X)$.  Notice
that if $\lambda_1 \neq \lambda_2$ then the choice of the eigenvalues
$\log \lambda_1$, $\log \lambda_2$ determines completely $D_0 X$.
\begin{definition}
  Consider the above setting.  We say that a resonance $x^i y^j e_1$ (resp.
  $x^i y^j e_2$) of $\psi$ is {\it strong} if the monomial
  $x^i y^j \partial / \partial x$ (resp. $x^i y^j \partial / \partial y$) is a
  resonant monomial of the vector field
  $\lambda_1 x \partial / \partial x + \lambda_2 y \partial / \partial y$, i.e
  if $\lambda_1^{i} \lambda_{2}^{j} = \lambda_1$ and
  $i \log \lambda_1 + j \log \lambda_2 = \log \lambda_1$ (resp.
  $\lambda_1^{i} \lambda_{2}^{j} = \lambda_2$ and
  $i \log \lambda_1 + j \log \lambda_2 = \log \lambda_2$).

  A resonance of $\psi$ that is not \emph{strong} will be called {\it weak}
\end{definition}

We will use the next particular case of \cite[Proposition 1.5]{ribon-jde-2012}.
\begin{proposition}
\label{pro:nonembedded}
Let $\psi \in \widehat{\mathrm{Diff}} ({\mathbb C}^{2},0)$ be such that
$(D_0 \psi) (x,y) =(\lambda_1 x, \lambda_2 y)$.  Let
$B: {\mathbb C}^{2} \to {\mathbb C}^{2}$ be a linear map such that
$\mathrm{exp}(B) = D_0 \psi$.  Assume that $j^k \psi = D_0 \psi + f_k$ where
both components of $f_k$ are homogeneous polynomials of degree $k$. Furthermore
assume that $f_k$ contains non-vanishing weakly resonant monomials.  Then $\psi$
is not embedded in the flow of any formal vector field $X$ such that
$D_0 X = B$.
\end{proposition} 
We can return to the problem of determining non-complete classes. 
The results regarding completeness of analytic classes ${\mathcal C}$
depend on the topology that we consider for the infinitely dimensional space ${\mathcal C}$.
First, we see that in some sense being connected by a geodesic is a dense property.
\begin{proof}[Proof of Proposition \ref{pro:natural}]
Let $\psi$ be a local diffeomorphism conjugating $\Gamma$ and $\Gamma'$.
Up to a linear change of coordinates we may assume that the linear part $D_0 \psi$
of $\psi$ at the origin is in Jordan normal form, in particular its matrix is of the form
   \begin{equation*}
    \left(
    \begin{array}{cc}
    u & w \\
    0 & v
    \end{array}
    \right)
  \end{equation*}
where  $u,v \in\mathbb{C}^{\star}$. Consider the family
$\sigma_{\epsilon} (x,y) = (e^{\epsilon a} x, e^{\epsilon b} y)$ for some
$a,b \in {\mathbb C}$ that are linearly independent over ${\mathbb Q}$.
The map $\sigma_\epsilon$ converges to $Id$ when $\epsilon \to 0$.
Let us define the family $(\Gamma_{\epsilon}')$ by
$\Gamma_{\epsilon}' = (\sigma_{\epsilon} \circ \psi) (\Gamma)$.
The map $D_0 (\sigma_{\epsilon} \circ \psi)$ has eigenvalues
$u e^{\epsilon a}$ and $v e^{\epsilon b}$.

Denote $F_{p,q} (\epsilon) = u^{p} v^{q} e^{(ap + bq) \epsilon} - 1$ and
$T_{p,q} = F_{p,q}^{-1} (0)$ for $(p,q) \in {\mathbb Z} \times {\mathbb Z}$.
Resonances between the eigenvalues of $D_0 (\sigma_{\epsilon} \circ \psi)$ are obtained when
there exists $(p,q) \in {\mathbb Z} \times {\mathbb Z} \setminus \{ (0,0) \}$
such that $(u e^{\epsilon a})^{p} (v e^{\epsilon b})^{q} = 1$.
This equation is equivalent to
$\epsilon \in T_{p,q}$. Since $a p + b q \neq 0$ the function $F_{p,q}$ is not constant and
$T_{p,q}$ is a countable closed set for any
$(p,q) \in {\mathbb Z} \times {\mathbb Z} \setminus \{ (0,0) \}$.
We deduce that $T:=\cup_{(p,q) \in {\mathbb Z} \times {\mathbb Z} \setminus \{ (0,0) \}} T_{p,q}$
is countable and hence
there exists $\epsilon_{0} \in {\mathbb C}^{*}$ such that
$\{ t \epsilon_0 : t \in {\mathbb R}^{*} \} \cap T = \emptyset$. We define the
path $\gamma:[0,\infty) \to {\mathbb C}$ by $\gamma(t) = t \epsilon_0$.
The map 
$\sigma_{\epsilon} \circ \psi$ is embedded in the flow of a formal vector field for any
$\epsilon \not \in T$  by
Corollary \ref{cor:Poincare-linear}.   Therefore $\Gamma$ and $\Gamma_{\epsilon}'$ 
are connected by a geodesic for any $\epsilon \in \gamma (0, \infty)$
by Theorem \ref{teo:formal-complete}.
 \end{proof}

 Let us show that the analytic class ${\mathcal C}_0$ of the plane branch
 $\Gamma_0$ with Puiseux parametrization $(t^6, t^7 + t^{10} + t^{11})$ is
 non-complete.  First, we study the stabilizer group
 $\mathrm{Stab} (\Gamma_0) = \{ \psi \in \widehat{\mathrm{Diff}} ({\mathbb
   C}^{2},0) : \psi (\Gamma_0)=\Gamma_0\}$ of $\Gamma_0$.
\begin{lemma}
\label{lem:linear-stab}
The linear part at the origin of any element $\psi$ of
$\mathrm{Stab} (\Gamma_0)$ is the identity map.
\end{lemma}
\begin{proof}
  The linear part $D_0 \psi$ is a map of the form
  $(x,y) \mapsto (ax + by, cx+ dy)$. Since $D_0 \psi$ preserves the tangent cone
  of $\Gamma_0$, we deduce $c=0$.  In particular $ad \neq 0$ because $\psi$ is a
  formal diffeomorphism.  We have
\begin{equation*}
   \psi (  t^6, t^7 + t^{10} + t^{11}) \equiv (a t^6 + b t^7 + b t^{10} + b t^{11} + O(t^{12}),
   d t^7 +d  t^{10} +d  t^{11}+ O(t^{12})).
\end{equation*}
Consider a formal power series $\sigma (t)$ such that
$\sigma (t)^{6} \equiv (x \circ \psi)(t^6, t^7 + t^{10} + t^{11})$.  It must
admit the expression
$\sigma (t) \equiv a^{1/6} t + (b/6) a^{-5/6} t^2 + O(t^{3})$.  Moreover, it is
a formal diffeomorphism in one variable and its inverse $\sigma^{-1}$ satisfies
$\sigma^{-1} (t) \equiv a^{-1/6} t - (b/6)a^{-4/3} t^2 + O(t^{3})$.  A simple
calculation leads us to
\begin{equation*}
   \psi ( t^6, t^7 + t^{10} + t^{11}) \circ \sigma^{-1} (t) \equiv \left( t^6,
   d a^{-7/6} t^7 - \frac{7}{6} b d a^{-7/3} t^{8} + O(t^{9}) \right).
\end{equation*}
Since $\psi$ belongs to $\mathrm{Stab} (\Gamma_0)$,  $bd  a^{-7/3}$ vanishes.
We deduce $b=0$ as a consequence of $ad \neq 0$.
Thus the formal diffeomorphisms $\sigma$ and $\sigma^{-1}$
are of the form $t \mapsto a^{1/6} t + O(t^7)$ and $t \mapsto a^{-1/6} t + O(t^7)$
respectively. We obtain
\begin{equation*}
   \psi ( t^6, t^7 + t^{10} + t^{11}) \circ \sigma^{-1} (t) \equiv (t^6,
 a^{-7/6}  d t^7 + a^{-10/6} d  t^{10} + a^{-11/6} d  t^{11}+ O(t^{12})).
\end{equation*}
Since $\psi (\Gamma_0)=\Gamma_0$ there exists $\xi \in {\mathbb C}$ such that
$\xi^{6} =1$ and $ a^{-7/6}  d = \xi^{7}$, $ a^{-10/6}  d = \xi^{10}$ and
$ a^{-11/6}  d = \xi^{11}$.
We get $a^{1/6} = \xi^{-1}$ by dividing the last two equations and then
$a  = (a^{1/6})^{6} = \xi^{-6} = 1$.
By plugging $a^{-1/6} = \xi$ into $ a^{-7/6}  d = \xi^{7}$, we get
$d=1$. Hence $D_0 \psi$ is the identity map.
\end{proof}
\begin{proposition}
\label{pro:stabx}
Let $X$ be a formal vector field that preserves $\Gamma_0$.
Then $X$ has vanishing second jet.
\end{proposition}
\begin{proof}
Since $X$ preserves $\Gamma_0$, it follows that 
the time $s$ flow  $\mathrm{exp} (s X)$ of $X$ preserves
$\Gamma_0$ for any $s \in {\mathbb C}$.  All the formal diffeomorphisms
$\mathrm{exp} (s X)$ in the one-parameter group of $X$ have identity linear part at
the origin by Lemma \ref{lem:linear-stab}.  In particular $X$ has vanishing linear part.
We write
\begin{equation*}
X = \left(  \sum_{i+j \geq 2} a_{ij} x^{i} y^{j} \right) \frac{\partial}{\partial x} +
\left(\sum_{i+j \geq 2} b_{ij} x^{i} y^{j} \right) \frac{\partial}{\partial y} .
\end{equation*}
Consider the dual form
\[
  \omega = - \left(\sum_{i+j \geq 2} b_{ij} x^{i} y^{j} \right) dx + \left(
    \sum_{i+j \geq 2} a_{ij} x^{i} y^{j} \right) dy .
\]
Since $X$ preserves
$\Gamma_0$, it follows that $(t^6, t^7 + t^{10} + t^{11})^{*} \omega \equiv 0$.
We have
\[ \nu_{\Gamma_0} (x^{2} dx) = 18, \  \nu_{\Gamma_0} (x y dx) = 19, \
 \nu_{\Gamma_0} (y^{2} dx) = 20, \]
\[ \nu_{\Gamma_0} (x^{2} dy) = 19, \ \nu_{\Gamma} (x y dy) = 20, \ \nu_{\Gamma} (y^{2} dy) = 21 \]
and $\nu_{\Gamma} (x^{i} y^{j} dx) \geq 24 \leq \nu_{\Gamma} (x^{i} y^{j} dy)$
for $i + j \geq 3$. Since
$(t^6, t^7 + t^{10} + t^{11})^{*} \omega \equiv 0$, we deduce $b_{20}=0$. We get
\[ - 6 b_{11} t^{11} (t^{7} + t^{10} + t^{11}) - 6 b_{02} t^{5} (t^{7} + t^{10} + t^{11})^{2} +  \]
\[ a_{20} t^{12} (7 t^{6} + 10 t^{9} + 11t^{10}) + a_{11} t^{6} (t^{7} + t^{10} + t^{11})
(7 t^{6} + 10 t^{9} + 11t^{10}) \]
\[ + a_{02} (t^{7} + t^{10} + t^{11})^{2}
(7 t^{6} + 10 t^{9} + 11t^{10}) + O(t^{23}) =0 . \]
We write the linear system of equations satisfied by the coefficients of
$t^{18}$, $t^{19}$, $t^{20}$, $t^{21}$ y $t^{22}$:
\[
\begin{array}{ccccccccccc}
- 6 b_{11} &  & & + & 7 a_{20} &   & & & & = & 0 \\
 & & -6 b_{02} &  & & + & 7 a_{11} & & & = & 0 \\
 &  & &  & &   & & + & 7 a_{02} & = & 0 \\
- 6 b_{11} &  & & + & 10 a_{20} &   & &  &  & = & 0 \\
- 6 b_{11} & - & 12 b_{02} & + & 11 a_{20} & +  & 17 a_{11} &  & & = & 0. \\
\end{array}
\]
The matrix of the system is regular, hence $b_{11}=b_{02}=a_{20}=a_{11}=a_{02}=0$.
In particular $X$ has a vanishing second jet.
\end{proof}
\begin{proposition}
\label{pro:2stab}
Let $\psi \in \mathrm{Stab}(\Gamma_0)$. Then $\psi$ and the identity map have the same
second jet.
\end{proposition}
\begin{proof}
The linear part of $\psi$ is the identity map by Lemma \ref{lem:linear-stab}.
Thus $\psi$ is of the form $\mathrm{exp} (X)$ for some unique formal nilpotent vector field $X$
(in fact $X$ has vanishing linear part) by Remark \ref{rem:nil-to-unip}.
Let $f =0$ be an irreducible equation of $\Gamma_0$.
Notice that $f \circ \mathrm{exp} (s X) = \sum_{j=0}^{\infty} \frac{s^j}{j!} X^j (f)$ 
by Taylor's formula and that 
$X^j (f) \in {\mathfrak m}^{j+1}$ for any $j \geq 1$. 
Therefore $f \circ \mathrm{exp} (s X)$ belongs to ${\mathbb C}[s][[x,y]]$ and then 
\begin{equation*}
G(s,t):= f \circ \mathrm{exp} (s X) \circ (t^6, t^7 + t^{10} + t^{11})
\end{equation*}
belongs to ${\mathbb C}[s][[t]]$. Moreover $G(s,t)$ vanishes for $s \in {\mathbb Z}$ since
$\{ \mathrm{exp} (s X) : s \in {\mathbb Z} \}$ is the cyclic group $\langle \psi \rangle$
and $\langle \psi \rangle$ is contained in $\mathrm{Stab}(\Gamma_0)$.
Since the coefficients of $t^j$ of $G(s,t)$ are polynomials that vanish at ${\mathbb Z}$,
we deduce $G \equiv 0$. In particular the elements of the one-parameter group generated by
$X$ preserve $\Gamma_0$ and hence $X$ preserves $\Gamma_0$.
By Proposition \ref{pro:stabx} the vector field $X$ has vanishing second jet and hence
$j^{2} \psi \equiv Id$.
\end{proof}
We just completed the first step of the proof of Proposition
\ref{pro:noncomplete}.  Now we want to construct $2$-jets of diffeomorphisms
such that any local diffeomorphism with such a $2$-jet is not embedded in the
flow of a formal vector field.
\begin{lemma}
\label{lem:nonembedded2}
Let $\psi \in \widehat{\mathrm{Diff}} ({\mathbb C}^{2},0)$ such that its second jet is equal to
$(x,y) \mapsto (x + x^{2} + y^{2}, -y)$. Then $\psi$ is not embedded in the flow of a formal vector
field.
\end{lemma}
\begin{proof}
Assume by contradiction that $\psi$ is of the form $\mathrm{exp}(X)$ for some formal nilpotent
vector field. The eigenvalues of the linear part of $X$ at the origin are $\alpha$ and $\beta$
with $e^{\alpha}=1$ and $e^{\beta}=-1$. We claim that for any choice of $\alpha$ and $\beta$
at least one of the resonances $x^{2} e_1$ or $y^{2} e_1$ is weak. Otherwise we obtain
\begin{equation*}
2 \alpha - \alpha =0 \ \mathrm{y} \ 2 \beta - \alpha=0 \implies \alpha=\beta=0 \implies e^{\beta}=1
\end{equation*}
and since $e^{\beta}=-1$ this is a contradiction. Hence, the formal
diffeomorphism $\psi$ is not embedded in a formal flow by Proposition
\ref{pro:nonembedded}.
\end{proof}
 \begin{lemma}
\label{lem:nonembedded3}
Let $\psi \in \widehat{\mathrm{Diff}} ({\mathbb C}^{2},0)$ be such that its
second jet is equal to
$(x,y) \mapsto (e^{2 \pi i/3} x + y^{2}, e^{4 \pi i/3} y + x^{2})$.  Then $\psi$
is not embedded in the flow of a formal vector field.
\end{lemma}
\begin{proof}
Assume by contradiction that $\psi$ is of the form $\mathrm{exp}(X)$ for some formal nilpotent
vector field. The eigenvalues of the linear part of $X$ at the origin are $\alpha$ and $\beta$
with $e^{\alpha}=e^{2 \pi i/3}$ and $e^{\beta}=e^{4 \pi i/3}$. We claim that for any choice of
$\alpha$ and $\beta$
at least one of the resonances $x^{2} e_2$ o $y^{2} e_1$ is weak. Otherwise we have
$2 \alpha = \beta$ and $2 \beta = \alpha$. 
This implies $\alpha=\beta=0$, contradicting $e^{\alpha}=e^{2 \pi i/3}$.
Therefore the formal diffeomorphism $\psi$ is not embedded
in a formal flow by Proposition \ref{pro:nonembedded}.
\end{proof}

\begin{proof}[Proof of Proposition \ref{pro:noncomplete}]
Consider the diffeomorphism
\begin{equation*}
\psi(x,y)=  (x + x^{2} + y^{2}, -y) \ \mathrm{or} \ \psi(x,y) = (e^{2 \pi i/3} x + y^{2}, e^{4 \pi i/3} y + x^{2})
\end{equation*}
and the curve $\Gamma = \psi (\Gamma_0)$. Any formal diffeomorphism
$\sigma$ conjugating $\Gamma_0$ and $\Gamma$ is of the form
$\psi \circ \rho$ where $\rho \in \mathrm{Stab} (\Gamma_0)$.
Since $j^2 \rho \equiv Id$ by  Proposition \ref{pro:2stab}, we deduce
$j^2 (\sigma \circ \rho) \equiv j^{2} \psi$ and hence
$\sigma \circ \rho$ is not embedded in the flow of a formal vector field for any
$\rho \in \mathrm{Stab} (\Gamma_0)$ by Lemmas \ref{lem:nonembedded2} and 
\ref{lem:nonembedded3}. Therefore the analytic class ${\mathcal C}_0$
is non-complete.
\end{proof}
\begin{proof}[Proof of Proposition \ref{pro:open-krull}]
Consider the subset $T$ of ${\mathrm{Diff}} ({\mathbb C}^{2},0)$
of diffeomorphisms whose second jet is equal to $(x + x^{2} + y^{2}, -y)$
(instead we could choose $(e^{2 \pi i/3} x + y^{2}, e^{4 \pi i/3} y + x^{2})$ too).
The set $T$ is open in the Krull topology.
Moreover, since $\mathrm{Stab} (\Gamma_0)$ consists of formal diffeomorphisms with
trivial second jet, it follows that $T$ is a union of left cosets of
$\mathrm{Diff}({\mathbb C}^{2},0)/ \mathrm{Stab} (\Gamma_0)$.
As a consequence its projection $\tilde{T}$ in
$\mathrm{Diff}({\mathbb C}^{2},0)/ \mathrm{Stab} (\Gamma_0) \sim {\mathcal C}_0$
is an open set in the induced quotient topology.
Every plane branch $\Gamma$ in $\tilde{T}$ is of the form $\sigma (\Gamma_0)$ where
$\sigma \in \mathrm{Diff}({\mathbb C}^{2},0)$ satisfies
$j^2 \sigma \equiv j^2 (x + x^{2} + y^{2}, -y)$.
Therefore $\Gamma_0$ is not connected to $\Gamma$ by a geodesic by the proof of
Proposition \ref{pro:noncomplete}.  We just obtained an open subset $\tilde{T}$ of
${\mathcal C}_0$ whose elements are not connected to $\Gamma_0$ by a geodesic.
\end{proof}

 \begin{remark}
Notice that in the examples in the proof of Proposition \ref{pro:noncomplete} the curves
$\Gamma_0$ and $\Gamma$ have the same tangent cone.
\end{remark}

\begin{remark}
  Let us focus in the case where $\psi(x,y)= (x + x^{2} + y^{2}, -y)$.
  Zariski's $\lambda$ invariant of $\Gamma_0$ is $\lambda=10$.  Let $\Gamma$ be
  the curve of parametrization
\begin{equation*}
\psi( t^6, t^7 + t^{10} + t^{11}) = (t^{6} + O(t^{12}), -t^{7} - t^{10} - t^{11} ).
\end{equation*}
Up to a change of parameter $t \mapsto ut$ with $u^{6}=1$, the curve $\Gamma$ is
of the form
\begin{equation*}
\left( t^{6} + O(t^{12}), -\frac{t^{7}}{u^{7}} - \frac{t^{10}}{u^{10}} - \frac{t^{11}}{u^{11}} \right)
\end{equation*}
and the coefficient of $t^7$ is equal to $1$ if and only if $u=-1$.
Then
\begin{equation*}
(t^{6} + O(t^{12}), t^{7} - t^{10} + t^{11} )
\end{equation*}
parametrises $\Gamma$. Every  parametrization
of $\Gamma$ of the form $(t^6, t^{7}+c t^{10} + O(t^{11}))$ satisfies $c=-1$. Thus the
curves $\Gamma_0$ and $\Gamma$ are not connected by a geodesic but
have the same tangent cone and their parametrizations
coincide up to (but not including) the term corresponding to Zariski's $\lambda$ invariant.
\end{remark}


\end{document}